\documentclass[12pt,oneside,reqno]{amsart}
\usepackage{enumerate}
\usepackage[hyperindex,breaklinks,colorlinks,citecolor=blue,pagebackref]{hyperref}
\usepackage[latin1]{inputenc}
\usepackage{amsmath, amsthm, amsfonts, amssymb}
\usepackage{mathrsfs}
\usepackage{graphicx}
\usepackage{latexsym}
\usepackage{fullpage}
\usepackage[all]{xy}
\usepackage{color}
\usepackage{stmaryrd}

\usepackage{latexsym}
\usepackage{enumerate}
\usepackage[hyperindex,breaklinks,colorlinks,citecolor=blue,pagebackref]{hyperref}
\usepackage[latin1]{inputenc}
\usepackage{soul}
\usepackage{tikz-cd}
\usepackage{caption}

\newtheorem{theorem}{Theorem}[section]
\newtheorem{lemma}[theorem]{Lemma}
\newtheorem{proposition}[theorem]{Proposition}
\newtheorem{corollary}[theorem]{Corollary}
\newtheorem{question}[theorem]{Question}

\newtheorem{definition}[theorem]{Definition}

\theoremstyle{remark}

\newtheorem{rmks}{Remarks}[section]
\newtheorem*{Examples}{Examples}
\newtheorem{example}{Example}

\newcommand{\N}{\mathbb{N}}

\renewcommand{\P}{\mathcal{P}}
\newcommand{\Q}{\mathcal{Q}}

\newcommand{\U}{\mathcal{U}}

\newcommand{\MLR}{\mathsf{MLR}}

\newcommand{\dom}{\text{dom}}

\newcommand{\llb}{\llbracket}
\newcommand{\rrb}{\rrbracket}

\newcommand{\fr}{{^{\smallfrown}}}

\definecolor{purple}{rgb}{.9,0.2,.9}

\newcommand{\cs}{2^\omega}
\newcommand{\uh}{{\upharpoonright}}
\renewcommand{\phi}{\varphi}
\newcommand{\str}{2^{<\omega}}
\newcommand{\binrat}{\mathbb{Q}_2}

\newcommand{\halts}{{\downarrow}}
\newcommand{\diverge}{{\uparrow}}

\renewcommand{\tt}{\mathrm{tt}}

\newcommand{\pz}{\Pi^0_1}

\newcommand{\atom}{\mathsf{Atoms}}

\usepackage{xcolor}	
	\usepackage{soul}

\title{The random members of a $\Pi^0_1$ class}
\author{Douglas Cenzer and Christopher P. Porter}
\date{} 

\begin{document}

\maketitle

\begin{abstract}We examine several notions of randomness for elements in a given $\Pi^0_1$ class $\P$. Such an effectively closed subset $\P$ of $\cs$ may be viewed as the set of infinite paths through the tree $T_\P$ of extendible nodes of $\P$, i.e., those finite strings that extend to a member of $\P$, so one approach to defining a random member of $\P$ is to randomly produce a path through $T_\P$ using a sufficiently random oracle for advice. In addition, this notion of randomness for elements of $\P$ may be induced by a map from $\cs$ onto $\P$ that is computable relative to $T_\P$, and the notion even has a characterization in term of Kolmogorov complexity. Another approach is to define a relative measure on $\P$ by conditionalizing the Lebesgue measure on $\P$, which becomes interesting if $\P$ has Lebesgue measure 0. Lastly, one can alternatively define a notion of incompressibility for members of $\P$ in terms of the amount of branching at levels of $T_\P$.  We explore some notions of homogeneity for $\Pi^0_1$ classes, inspired by work of van Lambalgen. A key finding is that in a specific class of sufficiently homogeneous $\Pi^0_1$ classes $\P$, each of these approaches coincides.  We conclude with a discussion of random members of $\Pi^0_1$ classes of positive measure.

\end{abstract}

\section{Introduction}
\label{intro}

The theory of algorithmic randomness for $2^\omega$\!, the collection of infinite binary sequences, has been actively studied over roughly the last fifteen years.  
During this time, there has been some work in considering algorithmic randomness in more general settings. Random closed subsets of $2^{\omega}$  were introduced by Barmpalias, Brodhead, Cenzer et al in \cite{BCDW07}, and also studied by Axon \cite{Axo15}, Diamondstone and Kjos-Hanssen \cite{DK09,DK12}, and others.  Random continuous functions were introduced by Barmpalias, Cenzer et al in \cite{BCRW09} and have recently been studied by Culver and Porter \cite{CP16}.  
Here, we consider only a slightly more general setting than $\cs$:  If we replace $\cs$ with some non-empty $\Pi^0_1$ class $\P\subseteq\cs$\!, that is, an effectively closed subset of $\cs$\!, is it reasonable to speak of the algorithmically random member of $\P$?  We are not simply asking here whether $\P$ contains any Martin-L\"of random members.  For according to two standard results in algorithmic randomness, if $\P$ has Lebesgue measure 0, it contains no Martin-L\"of random sequences, while if $\P$ has positive Lebesgue measure, it contains, up to finite modification, \emph{every} Martin-L\"of random sequence.

The approach we take here is to treat $\P$ as a space in its own right.  In particular, we consider three general approaches to defining the random members of $\P$:  
\begin{enumerate}
\item  a local approach, according to which the randomness of some $X\in \P$ is defined in terms of the behavior of initial segments of $X$ in the set of extendible nodes of $\P$, or equivalently, in terms of properties of the relatively clopen sets $\llb X\uh n\rrb\cap \P$ (in the subspace topology), 
\smallskip

\item a global approach, according to which the randomness of some $X\in \P$ is defined in terms of properties of the entire class $\P$; and

\smallskip

\item an intermediate approach, according to which the randomness of some $X\in \P$ is defined in terms of the properties of the relatively clopen sets $\llb 2^n\rrb\cap \P$, that is, clopen sets generated by the collection of strings of length $n$ that extend to an infinite path in $\P$.
\end{enumerate}


The key finding of this study is that in sufficiently homogeneous $\Pi^0_1$ classes, each of these approaches coincide.  We also consider a number of specific examples of $\Pi^0_1$ classes and investigate the properties of their random members.

The remainder of the paper proceeds as follows.  First, in Section \ref{Background} we provide the necessary background for the remainder of the study.  Next, in Section \ref{Local},
 we explore the local approach by formalizing the idea of randomly producing a random path of a fixed $\Pi^0_1$ class.  We also provide a number of equivalent characterizations of this initial definition.  In Section \ref{Global}, we define a global notion of randomness for members of a $\Pi^0_1$ class $\P$ by conditionalizing the Lebesgue measure on $\P$ and compare this to the local approach.   Next, we consider in Section \ref{sec-Intermediate} an intermediate approach first given by van Lambalgen (although he was not explicitly attempting to formalize the notion of random member of a $\Pi^0_1$ class).  This intermediate approach involves the Kolmogorov complexity of the initial segments of  elements of $\P$ in comparison with the number of initial segments of the same length.  This approach is compared and contrasted with the local and global approaches.  In Section \ref{sec-Homogeneous},  we examine various notions of homogeneous classes and show that for a sufficiently homogeneous $\Pi^0_1$ class $\P$, the local, global, and intermediate approaches to defining random members of $\P$ coincide. In Section \ref{sec-Positive}, we examine the randomness of members of classes of positive measure. Conclusions are given and prospects for future work in are discussed in Section \ref{sec-future}.

\section{Background} \label{Background}

We will assume the reader is familiar with the basics of computability theory.  We fix some notation and provide some basic definitions.

\subsection{Notation and preliminary definitions}
$\str$ denotes the collection of finite binary strings, members of which are denoted by lowercase Greek letters such as $\sigma$ and $\tau$.  For $n\in\omega$, $2^n$ denotes the collection of binary strings of length $n$.  Given a finite string $\sigma
\in 2^n$, let $|\sigma| = n$ denote the length of $n$.  For each $\sigma\in\str$ and each $i<|\sigma|$, $\sigma(i)$ denotes the $(i+1)^\mathrm{st}$ bit of $\sigma$. For two strings $\sigma,\tau$, we say that $\tau$ \emph{extends} $\sigma$
and write $\sigma \preceq \tau$ if $|\sigma| \leq |\tau|$ and $\sigma(i)
= \tau(i)$ for $i < |\sigma|$.  We say that $\sigma$ and $\tau$ are \emph{compatible} if either $\sigma \preceq \tau$ or $\tau \preceq \sigma$. 
 Let $\sigma^{\frown} \tau$
denote the concatenation of $\sigma$ and $\tau$; we will often write, for instance, $\sigma0$ and $\sigma1$ instead of $\sigma\fr 0$ and $\sigma\fr 1$.  The empty string will be denoted $\epsilon$.  For $\sigma\in\str$, we define $\sigma^\curvearrowright$ to be the string obtained by flipping the last bit of $\sigma$, so that if $\sigma=\tau\fr i$, then $\sigma^\curvearrowright=\tau\fr(1-i)$.

 $\cs$ denotes the collection of infinite binary
sequences, which we will often identify with subsets of $\omega$, viewing each sequence as a characteristic function. We will write members of $\cs$ as uppercase Roman letters such as $X,Y$ and $Z$.  We will write the complement of $X\in\cs$ as $\overline X$.  For $X \in \cs$, $\sigma \prec X$ means
that $\sigma(i) = X(i)$ for $i < |\sigma|$, where as above, $X(i)$ denotes the $(i+1)^\mathrm{st}$ bit of $x$.  Let $X
\uh n = X(0)\dots X(n-1)$; $\sigma\uh n$ is similarly defined for $\sigma\in\str$ and $n<\sigma$.  
Two sequences $X$ and $Y$ may be coded together into $Z = X \oplus Y$, where 
$Z(2n) = X(n)$ and $Z(2n+1) = Y(n)$ for all $n$.  More generally, given any $W\in\cs$, we can code $X$ and $Y$ together via $W$ using the principal function of $W$, i.e., the function $p_W$ such that $p_W(n)=i$ if and only if $i$ is the $(n+1)^\mathrm{st}$ number such that $W(i)=1$, and that of $\overline{W}$.  Then we define $Z=X\oplus_W Y$ by
\[
Z(m)=
\left\{
	\begin{array}{ll}
		X(n)  & \mbox{if } p_{W}(n)=m \\
		Y(n) & \mbox{if } p_{\overline{W}}(n)=m
	\end{array}.
\right.\]
That is, $X\oplus_W Y$ is obtained by coding $X$ at locations $n$ where $W(n)=1$ and coding $Y$ at locations $n$ where $W(n)=0$.

For a finite string $\sigma$, let $\llb\sigma\rrb$ denote $\{X \in \cs:
 \sigma \prec X \}$. We shall refer to $\llb\sigma\rrb$ as the \emph{interval}
 determined by $\sigma$. Each such interval is a clopen set and the
 clopen sets are just finite unions of intervals. 
A set $T\subseteq\str$ is a tree if it is closed downwards under~$\preceq$.
For $n \in \omega$, let $T \uh n = T \cap \{0,1\}^n$. 
A nonempty closed set $\P$ may be identified with a tree $T_\P
 \subseteq \str$ where $T_\P = \{\sigma: \P \cap \llb\sigma\rrb \neq
 \emptyset\}$  Note that $T_\P$ has no dead ends. That is, if $\sigma
 \in T_\P$, then either $\sigma^{\frown}0 \in T_\P$ or $\sigma^{\frown}1
 \in T_\P$ (or both). Note that $T_\P \uh n = \{X \uh n: X \in \P\}$. 
For an arbitrary tree $T \subseteq \str$, let $[T]$ denote the
set of infinite paths through $T$, i.e., the set of $X$ such that $X\uh n\in T$ for every $n\in\omega$.
It is well known that $\P \subseteq \cs$ is a closed set if and only if
$\P = [T]$ for some tree $T$.  $\P$ is a $\Pi^0_1$ \emph{class}, or an effectively
closed set, if $\P = [T]$ for some computable tree $T$. 
 The complement of a $\Pi^0_1$ class is said to be a \emph{c.e.\ open set}. This notion plays an important role in 
algorithmic randomness and provides a link between the two areas of study: effectively closed sets and algorithmic randomness. 
Note that in general, if $\P$ is a $\pz$ class, then $T_\P$ is a $\Pi^0_1$ set  but $T_\P$ need not be computable.
If $T_\P$ is computable, then $\P$ is said to be a \emph{decidable} $\pz$ class. 
Moreover, a nonempty $\pz$ class need not contain any computable elements, but a nonempty decidable class
certainly contains computable elements, for  example the left- and right-most infinite paths. Given two closed sets $\P$ and $\Q$, we can define the product $\P \otimes \Q = \{X \oplus Y: X \in \P\ \&\ Y \in \Q\}$ and the disjoint union $\P \oplus \Q = \{0 \fr X: X \in \P\} \cup \{ 1 \fr Y: Y \in \Q\}$; if $\P$ and $\Q$ are $\pz$ classes, then $\P \otimes \Q$ and $\P \oplus Q$ are also $\pz$ classes.  
For a detailed development of $\Pi^0_1$ classes, see \cite{CRta,CR98}.

\subsection{Computable measures and Turing functionals}

By Caratheodory's Theorem, a measure $\mu$ on $\cs$ is uniquely determined by specifying the values of $\mu$ on the intervals of $\cs$, where $\mu(\llb\sigma\rrb)=\mu(\llb\sigma0\rrb)+\mu(\llb\sigma1\rrb)$ for every $\sigma\in\str$.  If in addition we require that $\mu(\cs)=1$, then $\mu$ is a \emph{probability measure}. Hereafter, we will write $\mu(\llb\sigma\rrb)$ as $\mu(\sigma)$. 

The \emph{Lebesgue measure}~$\lambda$ is the unique Borel measure such that $\lambda(\sigma)=2^{-|\sigma|}$ for all $\sigma\in\str$.  A measure $\mu$ on $\cs$ is \emph{computable} if $\sigma \mapsto \mu(\sigma)$ is computable as a real-valued function, i.e., if there is a computable function $\tilde \mu:\str\times\N\rightarrow\binrat$  (where $\binrat=\{\frac{m}{2^n}:n,m\in\omega\}$) such that $|\mu(\sigma)-\tilde \mu(\sigma,i)|\leq 2^{-i}$ for every $\sigma\in\str$ and $i\in\omega$.    
This notion can be relativized to any oracle $A$, yielding an $A$-computable measure.

One family of computable measures is given by the collection of measures that are concentrated on a single computable point.  If $X\in\cs$ is a computable sequence, then the \emph{Dirac measure concentrated on $X$}, denoted $\delta_X$, is defined as follows:
\[
\delta_X(\sigma)=
\left\{
	\begin{array}{ll}
	1  & \mbox{if } \sigma\prec X\\
	0 & \mbox{if } \sigma\not\prec X
	\end{array}.
\right.
\]
More generally, for a measure $\mu$, we say that $X\in\cs$ is an \emph{atom of $\mu$} or a \emph{$\mu$-atom}, denoted $X\in\atom_\mu$, if $\mu(\{X\})>0$.  Kautz \cite{Kau91} fully characterized the atoms of a computable measure, showing that
$X\in\cs$ is computable if and only if $X$ is an atom of some computable measure.

There is a close connection between computable measures and 
 a certain class of Turing functionals. Recall that a continuous function $\Phi: \cs \to \cs$ may be defined from a function $\phi: \str \to \str$ such that

\begin{enumerate}
\item[(i)]$\sigma \prec \tau$, then $\phi(\sigma) \prec \phi(\tau)$, and 

\item[(ii)] For all $X \in \cs$, $\lim_n |\phi(X \uh n| = \infty$.
\end{enumerate}

\noindent Moreover, a representation $\phi$ for a continuous function must satisfy:

\begin{enumerate}
\item[(iii)] For all $m$, there exists $n$ such that for every $\sigma \in \{0,1\}^n$, $|\phi(\sigma)| \geq m$. 
\end{enumerate}

\noindent We then have $\Phi(X) = \bigcup_n \phi(X \uh n)$. The (total)  Turing functionals $\Phi: \cs \to \cs$ are those which may be defined in this manner from computable $\phi: \str \to \str$.  We will sometimes refer to total Turing functionals as \emph{$\tt$-functionals.} The partial Turing functionals $\Phi:\subseteq\cs\rightarrow\cs$ are given by those $\phi: \str \to \str$ which only satisfy condition (i). In this case $\Phi(X) = \bigcup_n \phi(X \uh n)$
may be ony a finite string.  We set $\dom(\Phi) = \{X: \Phi(X) \in \cs\}$.  For $\tau \in \str$ let $\Phi^{-1}(\tau)$ be defined by
\[
\Phi^{-1}(\tau)=\{ \sigma\in \str : \tau \preceq \phi(\sigma)\;\&\;(\forall \sigma'\prec\sigma)\; \tau\not\preceq\phi(\sigma')\}.  
\]
In particular, by our above convention, we have $\Phi^{-1}(\epsilon)=\{\epsilon\}$. Similarly, for $S \subseteq \str$ we define $\Phi^{-1}(S) = \bigcup_{\tau \in S} \Phi^{-1}(\tau)$. When $\mathcal{A}$ is a subset of $\cs$, we denote by $\Phi^{-1}(\mathcal{A})$ the set $\{X\in \dom(\Phi):\Phi(X) \in \mathcal{A}\}$. Note in particular that $\Phi^{-1}(\llb\tau\rrb) = \llb\Phi^{-1}(\tau)\rrb \cap \dom(\Phi)$.\\

The Turing functionals that induce computable measures are precisely the \emph{almost total} Turing functionals, where a Turing functional $\Phi$ is almost total if $\lambda(\dom(\Phi))=1$. Given an almost total Turing functional $\Phi$, the measure induced by $\Phi$, denoted $\lambda_\Phi$, is defined by
\[
\lambda_\Phi(\sigma)=\lambda(\llb\Phi^{-1}(\sigma)\rrb)
=\lambda(\{X:\Phi^X\succeq\sigma\}).
\]
It is not difficult to verify that $\lambda_\Phi$ is a computable measure. That is, for any string $\sigma$, $\lambda_{\Phi}(\sigma)$ is the limit of the computable increasing sequence $r_n(\sigma) =  2^{-n} \cdot \#\{\tau \in 2^n: \sigma \prec f(\tau)\}$.  On the other hand, since the domain of $\Phi$ has measure 1, we also have $\lambda_{\Phi}(\sigma) = 1 - \sum\{\lambda_{\Phi}(\rho): \rho \in 2^{|\sigma|} \setminus \{\sigma\}\}$.  This shows that $\lambda_{\Phi}(\sigma)$ is also the limit of a computable decreasing sequence, and hence is computable. 

 Moreover, one can easily show that given a computable probability measure $\mu$, there is some almost total functional~$\Phi$ such that $\mu=\lambda_\Phi$. 

Given $A\in\cs$, a function $\Phi:\cs\rightarrow\cs$ is an $A$-computable functional if $\Phi$ is defined in terms of an $A$-computable functional $\phi:\str\rightarrow\str$ as above.  If $\Phi$ is total and $\Phi(X)=Y$, we write $Y\leq_{\tt(A)}X$.


\subsection{Algorithmic randomness}
Martin-L\"{o}f \cite{ML66} observed that stochastic properties could
be viewed as special kinds of effectively presented measure zero sets and defined a random
real as one that avoids these measure zero
sets. More precisely, a sequence $X \in \cs$ is Martin-L\"{o}f random if for every
effective sequence $\U_1, \U_2, \dots$ of c.e.\ open sets with $\lambda(\U_n)
\leq 2^{-n}$, $X \notin \bigcap_n \U_n$.   This can be straightforwardly extended to any computable measure 
$\mu$ on $\cs$ by replacing the condition $\lambda(\U_n)
\leq 2^{-n}$ with $\mu(\U_n)\leq 2^{-n}$.  For a computable measure $\mu$, the collection of $\mu$-Martin-L\"of random sequences will be denoted $\MLR_\mu$; in the case that $\mu=\lambda$, we will simply write this collection as $\MLR$.  Martin-L\"of also proved the existence of a universal test $(\hat{\U_i})_{i\in\omega}$, so that $X\in\MLR$ if and only if $X\notin\bigcap_{i\in\omega}\hat{\U_i}$.

%

If $U:\str\rightarrow\str$ is a universal prefix-free machine (that is, for $\sigma,\tau\in\str$, if $\sigma\prec\tau$ and $U(\sigma)\halts$, then $U(\tau)\diverge$ and for every prefix-free machine $M$ there is some $e\in\omega$ such that $U(1^e0\sigma)=M(\sigma)$), then we define the \emph{prefix-free Kolmogorov complexity} of $\sigma\in\str$ to be ${K(\sigma)=\min\{|\tau|\colon U(\tau)\halts=\sigma\}}$.  One of the central results in algorithmic randomness is the following.

\begin{theorem}[Levin/Schnorr]\label{thm-levinschnorr}
For each computable measure $\mu$ and each $X\in\cs$, $X$ is $\mu$-Martin-L\"of random if and only if there is some $c$ such that for all $n$, 
\[
K(X\uh n)\geq -\log(\mu(X\uh n))-c.
\]
\end{theorem}
We will often use the relativized version of the Levin-Schnorr theorem throughout the paper.  

The following results will feature prominently in this study:

\begin{theorem} \label{thm-randpres}  Let $\Phi:\cs\rightarrow\cs$ be an almost total Turing functional.
\begin{itemize}
\item[(i)]  If $X\in\MLR$ then $\Phi(X)\in\MLR_{\lambda_\Phi}$.
\item[(ii)] If $Y\in\MLR_{\lambda_\Phi}$, then there is some $X\in\MLR$ such that $\Phi(X)=Y$.
\end{itemize}
\end{theorem}

We will refer to (i) as the \emph{preservation of randomness}, while (ii) will be referred to as the \emph{no randomness ex nihilo principle}.  We will also use relativized versions of these two results.

\section{A local approach to defining randomness in a $\Pi^0_1$ class} \label{Local}

As described in the introduction, on the local approach to defining the random members of a given $\Pi^0_1$ class $\P$, the randomness of some $X\in \P$ is determined by the behavior of initial segments of $X$ in the set of extendible nodes of $\P$.  There are several ways to make this precise.  The first can be seen as a formalization of the idea of randomly producing a path through $T_\P$, the set of extendible nodes of $\P$.

\subsection{Randomly produced paths through $T_\P$ }

Suppose we would like to randomly produce a path through $T_\P$ for a given $\Pi^0_1$ class $\P$.  We can use the following procedure to construct a path.  Having constructed a string $\sigma\in T_\P$ thus far, if $\sigma$ has only one extension in $T_\P$, we have no choice but to follow that extension.  However, if both $\sigma0$ and $\sigma1$ are in $T_\P$, then we toss an unbiased coin to determine which of these extensions our path will pass through.

In the context of algorithmic randomness, we can formalize this procedure as follows.   Given a Martin-L\"of random $R\in\cs$\!, suppose that we have produced an initial segment  $\sigma\in T_\P$ of a path in $\P$, having thus far used $R\uh n$ as advice for some $n\in\omega$.  If $\sigma\fr i\in T_\P$ but $\sigma\fr (1-i)\notin T_\P$ for $i\in\{0,1\}$, we have no choice but to extend $\sigma$ to $\sigma\fr i$.  However, if both $\sigma\fr 0\in T_\P$ and $\sigma\fr1\in T_\P$, then we extend $\sigma$ to $\sigma\fr R(n)$.  

The resulting path will have the form $X=R\oplus_{B_X}\!N_X$, where 
$B_X$ codes the levels where there is branching in $T_\P$ along $X$ and $N_X$ consists of the values that are not determined by $R$, which occur at non-branching levels of $T_\P$ along $X$.  More precisely, we have
\begin{itemize}
\item[(i)] $B_X=\{n:(X\uh n)^\curvearrowright\in T_\P\}$, and
\item[(ii)] if $\overline{ B_X}=\{n_0<n_1<\dotsc\}$, then for each $k\in\omega$, $X\uh n_k$ has only one extension in $T_\P$, namely, $(X\uh n_k)\fr N_X(k)$ (where $X\uh 0=\epsilon$).
\end{itemize}
Observe that $B_X,N_X\leq_T R\oplus T_\P$.  In fact, we have $B_X,N_X\leq_{tt(T_\P)}R$, which implies that there is a total $T_\P$-computable functional $\Psi_\P$ such that
$\Psi_\P(R)=R\oplus_{B_X}\!N_X=X$.  We now define the randomly produced paths through $T_\P$ as follows.

\begin{definition}
For a $\Pi^0_1$ class $\P$ with tree of extendible nodes $T_\P$, the \emph{randomly produced paths through $T_\P$} are given by the set $\{\Psi_\P(R): R\in\MLR^{T_\P}\}$.
\end{definition}

Let us consider several examples.

\begin{example}\label{yy}{\ }
\begin{enumerate}
\item Let $\P=\{X\oplus X:X\in\cs\}$.  Then the randomly produced paths through $T_\P$ 
are all sequences of the form $R\oplus R$ for $R\in\MLR$.  
\item  More generally, for $n\in\omega$, let $\P=\{\bigoplus_{i=1}^n X:X\in\cs\}$.  Then the randomly produced paths through $T_\P$ 
are all sequences of the form $\bigoplus_{i=1}^n R$ for $R\in\MLR$. 
\item Let $\P$ be a $\Pi^0_1$ class that contains an isolated point $X$.  Then $X$ is a randomly produced path through $T_\P$.
\end{enumerate}
\end{example}

This latter example appears to be an undesirable consequence of the definition of a randomly produced path through $T_\P$ for a $\Pi^0_1$ class $\P$, since the isolated points of a $\Pi^0_1$ class are computable.  However, we take a randomly produced path through $T_\P$ to be one produced by tossing a coin whenever we arrive at a branching node in $T_\P$, counting isolated points as random is consistent with this definition.

Let $\lambda_{\Psi_\P}$ be the $T_\P$-computable measure induced by $\Psi_\P$.  Then we have:
\begin{proposition}\label{prop-rpp}
Let $\P$ be a $\Pi^0_1$ class.  For $X\in\cs$, $X$ is a randomly produced path through $T_\P$ if and only if $X\in\MLR_{\lambda_{\Psi_\P}}^{T_\P}$.
\end{proposition}

\begin{proof}
($\Rightarrow$)  If $X$ is a randomly produced path through $T_\P$, then $X=\Psi_\P(R)$ for some $R\in\MLR^{T_\P}$. By the preservation of randomness theorem, it follows that $X\in\MLR_{\lambda_{\Psi_\P}}^{T_\P}$.\\
($\Leftarrow$)  If $X\in\MLR_{\lambda_{\Psi_\P}}^{T_\P}$, then by the no randomness ex nihilo principle, there is some $R\in\MLR^{T_\P}$ such that  $X=\Psi_\P(R)$, and hence $X$ is a randomly produced path through $T_\P$.
\end{proof}

Observe that the complexity of the measure $\lambda_{\Psi_\P}$ is determined by the complexity of the underlying tree of extendible nodes $T_\P$.  Thus, if $\P$ is decidable, then $T_\P$ is computable and hence so is $\lambda_{\Psi_\P}$.  However, if $\P$ is undecidable, then $T_\P$ is co-c.e.\ but not computable, from which it follows that $\lambda_{\Psi_\P}$ is merely $\emptyset'$-computable and not computable.  Although algorithmic randomness with respect to a computable measure has been well-studied, the topic of algorithmic randomness with respect to an $\emptyset'$-computable measure has not been treated systematically.

\subsection{An alternative characterization}

Observe that if $X$ is a randomly produced path through $T_\P$ and is not isolated in $\P$, then there is a unique $R\in\MLR^{T_\P}$ such that $\Psi_\P(R)=X$.  We can obtain an equivalent characterization of the randomly produced paths through $T_\P$ by instead considering a $T_\P$-computable functional $\Phi_\P$ that maps $\cs$ onto $\P$ and does not necessarily satisfy the above uniqueness condition on random sequences.

We define $\Phi_\P:\cs\rightarrow\P$ by means of a $T_\P$-computable function $\phi:\str\rightarrow T_\P$, which we define inductively as follows.  First, we set $\phi(\epsilon)=\epsilon$.  Next, suppose $\phi(\sigma)$ is defined for every $\sigma\in 2^n$.  Then to define $\phi(\sigma\fr i)$ for $i\in\{0,1\}$, we have two cases to consider.  
\begin{itemize}
\item[] \emph{Case 1}:  If $\phi(\sigma)\fr i\in T_\P$ for each $i\in\{0,1\}$, then for each such $i$ we set $\phi(\sigma  i)=\phi(\sigma)\fr i$.
\item[] \emph{Case 2}:  If $\phi(\sigma)\fr i\notin T_\P$ for some $i\in\{0,1\}$ (so that $\phi(\sigma)\fr (1-i)\in T_\P$, then we set $\phi(\sigma 0)=\phi(\sigma 1)=\phi(\sigma)\fr (1-i)$.
\end{itemize}
%
For $X\in\cs$, we define $\Phi_\P(X)=\bigcup_{n\in\omega}\phi(X\uh n)$.  Note that $\Phi_\P$ is a total $T_\P$-computable functional, and hence by Theorem \ref{thm-randpres}(i) it induces a $T_\P$-computable measure, which we will write as $\mu_\P$.

We now prove the following.

\begin{theorem}\label{thm-2equiv}
Let $\P$ be a $\Pi^0_1$ class.  Then $X\in\cs$ is a randomly produced path through $T_\P$ if and only if $X$ is $T_\P$-Martin-L\"of random with respect to $\mu_\P$.
\end{theorem}

\begin{proof}
We show inductively that $\Phi_\P$ and $\Psi_\P$ induce the same $T_\P$-computable measure.  First, we have $\Phi_\P^{-1}(\epsilon)=\Psi_\P^{-1}(\epsilon)=\cs$, and hence $\lambda_{\Phi_\P}(\epsilon)=1=\lambda_{\Psi_\P}(\epsilon)$.  Next, suppose that $\lambda_{\Phi_\P}(\sigma)=\lambda_{\Psi_\P}(\sigma)$ for every $\sigma$ of length $n$.  For a fixed $\sigma\in2^n$, let
\[
k=\#\{i\leq n: (\sigma\uh i)^\curvearrowright\in T_\P\}.
\]
That is, $k$ is the number of initial segments of $\sigma$ that are branching nodes in $T_\P$, i.e., strings $\tau$ such that $\tau0$ and $\tau1$ are in $T_\P$.  For our fixed $\sigma$, it follows from the definition of $\Psi_\P$ that there is some $\tau\in\str$ with $|\tau|=k$ such that $\Psi^{-1}_\P(\sigma)=\{\tau\}$.  Hence $\lambda_{\Psi_\P}(\sigma)=2^{-k}$.  Similarly, there are $\tau_1,\dotsc,\tau_j$ such that $\Phi^{-1}_\P(\sigma)=\{\tau_1,\dotsc,\tau_j\}$, so that $\lambda_{\Phi_\P}(\sigma)=\sum_{\ell=1}^j2^{-|\tau_\ell|}=2^{-k}$.  We now consider two cases:
\begin{itemize}
\item[] \emph{Case 1}:  Suppose that $\sigma\fr  i\in T_\P$ for each $i\in\{0,1\}$. Then:
\begin{itemize}
\item  For each $i\in\{0,1\}$, by definition of $\Psi_\P$, we set $\Psi^{-1}_\P(\sigma\fr  i)=\{\tau\fr  i\}$. 
\item  For each $i\in\{0,1\}$, by definition of $\Phi_\P$, we set $\Phi^{-1}_\P(\sigma\fr  i)=\{\tau \fr i:\tau\in \Phi^{-1}_\P(\sigma)\}$.
\end{itemize}
It follows that for each $i\in\{0,1\}$,
\[
\lambda_{\Psi_\P}(\sigma\fr i)=2^{-(k+1)}=\frac{1}{2}\sum_{\ell=1}^j2^{-|\tau_\ell|}=\sum_{\ell=1}^j2^{-|\tau_\ell|+1}=\sum_{\ell=1}^j2^{-|(\tau_\ell)\fr i|}=\lambda_{\Phi_\P}(\sigma\fr i).
\]
\item[] \emph{Case 2}:  Suppose that for some $i\in\{0,1\}$, $\sigma\fr  i\notin T_\P$ and $\sigma\fr (1-i)\in T_\P$.  Then:
\begin{itemize}
\item  By definition of $\Psi_\P$, we set $\Psi^{-1}_\P(\sigma\fr(1-i))=\{\tau\}$. 
\item  By definition of $\Phi_\P$, we set 
\[
\Phi^{-1}_\P(\sigma\fr(1-i))=\{\tau  0:\tau\in \Phi^{-1}_\P(\sigma)\}\cup\{\tau  1:\tau\in \Phi^{-1}_\P(\sigma)\}.
\]
\end{itemize}
\end{itemize}
It follows that
\begin{align*}
\lambda_{\Psi_\P}(\sigma\fr(1-i))=2^{-k}&=\sum_{\tau\in \Phi^{-1}_\P(\sigma)}\lambda(\tau)\\
&=\sum_{\tau\in \Phi^{-1}_\P(\sigma)}\lambda(\tau 0)+\sum_{\tau\in \Phi^{-1}_\P(\sigma)}\lambda(\tau 1)\\
&=\lambda_{\Phi_\P}(\sigma\fr(1-i))
\end{align*}
and hence
\[
\lambda_{\Psi_\P}(\sigma\fr i)=\lambda_{\Phi_\P}(\sigma\fr i)=0.
\]
It follows by induction that $\lambda_{\Psi_\P}=\lambda_{\Phi_\P}$.  Hence by Proposition \ref{prop-rpp}, $X\in\cs$ is a randomly produced path through $T_\P$ if and only if $X\in\MLR^{T_\P}_{\lambda_{\Psi_\P}}=\MLR^{T_\P}_{\lambda_{\Phi_\P}}$.
\end{proof}
%
%

\subsection{Initial segment complexity characterization of randomly produced paths}\label{subsec-isc1}

We can also characterize the randomly produced paths through $T_\P$ for a given $\Pi^0_1$ class $\P$ in terms of initial segment complexity.  To provide such a characterization, we need to identify a threshold so that the collection of sequences in $\P$ whose initial segment complexity is above this threshold is precisely the collection of randomly produced paths through $T_\P$.  One possibility is to define $X\in\P$ to be incompressible in $\P$ if
\[
K(X\uh n)\geq \log\#T_\P\uh n-O(1),
\]
a notion considered by van Lambalgen in \cite{Van87}.  However, in general, the collection of members of $\P$ satisfying this definition of incompressibility does not agree with the collection of randomly produced paths through $T_\P$, since $\log\#T_\P\uh n$ measures how much branching has occurred in $T_\P$ up to strings of length $n$, but it does not take into consideration how this branching (or lack of branching) is distributed among these strings.  In particular, such an approach is at odds with the local approach to defining randomness in $\P$ that we are considering here.  We will look more closely at a related notion of incompressibility in Section \ref{sec-Intermediate}.

In order to better capture the local structure of $T_\P$ along initial segments of a sequence $X\in\P$, we propose the following threshold.  For $\sigma\in\str$, we define 
\[
\theta_\P(\sigma)=\#\{n: 1\leq n\leq |\sigma|\;\&\;(\sigma\uh n)^\curvearrowright\in T_\P\},
\]
i.e., $\theta_\P(\sigma)$ is the number of initial segments of $\sigma$ whose siblings are extendible (a quantity used in the proof of Theorem \ref{thm-2equiv}).  Note that $\theta_\P\equiv_T T_\P$, and hence $\theta_\P$ is a computable function if and only if $\P$ is a decidable $\Pi^0_1$ class. 

Let us consider several examples of $\theta_\P$ for various $\Pi^0_1$ classes $\P$.

\begin{example}{\ }
\begin{enumerate}

\item If $\P=\cs$, so that $T_\P=\str$, then $\theta_\P(\sigma)=|\sigma|$ for every $\sigma\in\str$.  Thus, we have $K(X\uh n)\geq n-O(1)$ for every random path in $\cs$, which agrees with the standard definition of Martin-L\"of randomness.\\

\item If $\P=\{00,11\}^\omega$, so that $T_\P=\{00,11\}^{<\omega}$, then for every $\sigma\in\{00,11\}^{<\omega}$, we have $\theta_\P(\sigma)=|\sigma|/2$, while for every $\sigma\in T_\P$ of the form $\tau\fr i$ for some $\tau\in\{00,11\}^{<\omega}$ and $i\in\{0,1\}$, we have $\theta_\P(\sigma)=(|\sigma|+1)/2$.  Thus, the incompressible members of $\P$ are those $X$ satisfying
\[
K(X\uh n)\geq n/2-O(1),
\]
which agrees with the notion of Martin-L\"of randomness on $\{00,11\}^\omega$ as discussed in part 1 of Example \ref{yy} above.\\

\item Let $\P$ be a $\Pi^0_1$ class with an isolated point $X$.  Then there is some least $k$ such that $\llb X\uh k\rrb\cap\P=\{X\}$.  Thus there is some $j\leq k$ such that for every $n\geq k$, $\theta_\P(X\uh n)=j$.  Since $K(n)\geq j-O(1)$ for every $n$, it follows that
\[
K(X\uh n)=K(n)-O(1)\geq j-O(1)=\theta_\P(X\uh n)-O(1)
\]
for every $n$.  Thus, isolated points in $\P$ satisfy the definition of incompressibility.
\end{enumerate}
\end{example}

These examples can be derived from the following result.

\begin{theorem}\label{thm-rpp-isc}
Let $\P$ be a $\Pi^0_1$ class, and let $\theta_\P$ be as above.  Then $X\in\P$ is a randomly produced path through $T_\P$ if and only if 
\[
K^{T_\P}(X\uh n)\geq \theta_\P(X\uh n)-O(1).
\]
\end{theorem}

\begin{proof}
The key observation to prove this theorem comes from comparing the function $\theta_\P$ to the measure $\lambda_{\Phi_\P}$ as defined in the previous subsection.  Writing $\lambda_{\Phi_\P}$ as $\mu$, we have $\mu(\tau)=\lambda(\sigma)$, where $\phi(\sigma)=\tau$ for the $T_\P$-computable function $\phi:\str\rightarrow T_\P$ used in the definition of $\Phi_\P$.  By the definition of $\phi$, this means that $\tau$ has $|\sigma|$ initial segments whose siblings are extendible, i.e.\ $\theta_\P(\tau)=|\sigma|$.  Thus it follows that 
\[
\mu(\tau)=\lambda(\sigma)=2^{-|\sigma|}=2^{-\theta_\P(\tau)}.
\]
Now by Theorem \ref{thm-2equiv}, the randomly produced paths through $T_\P$ are precisely the $T_\P$-Martin-L\"of random members of $\cs$ with respect to $\mu$.  By the Levin-Schnorr Theorem for $T_\P$-Martin-L\"of randomness with respect to $\mu$, it follows that $X$ is a random path in $\P$ if and only if
\[
K^{T_\P}(X\uh n)\geq-\log\mu(X\uh n)-O(1)=\theta_\P(X\uh n)-O(1).
\]
\end{proof}

\section{A global approach to defining randomness in a $\Pi^0_1$ class} \label{Global}

We now turn to a global definition of the random members of a fixed $\Pi^0_1$ class $\P$, according to which the randomness of some $X\in \P$ is defined in terms of properties of the entire class $\P$.  The definition we provide here will be obtained by considering the Lebesgue measure conditional to $\P$.

\subsection{Conditionalizing the Lebesgue measure}

Let us first consider the case that $\P$ is a $\Pi^0_1$ class of positive Lebesgue measure.  For $\sigma\in\str$, we define the \emph{relative measure of $\llb\sigma\rrb$ in $\P$}
is 
\[
\lambda(\sigma\mid\P)=\dfrac{\lambda(\llb\sigma\rrb\cap\P)}{\lambda(\P)}.
\]
If $\lambda(\P)=0$, we clearly need an alternative definition.  Thus, we have the following definition.

\begin{definition}
Let $\P$ be a $\Pi^0_1$ class.  
\begin{enumerate}
\item The \emph{limiting relative measure of $\llb\sigma\rrb$ in $\P$} is
\[
\lambda_\P(\sigma)=\lim_{n\rightarrow\infty}\dfrac{\lambda(\llb\sigma\rrb\cap \llb T_\P\uh n\rrb)}{\lambda(T_\P\uh n)}.
\]

\item The \emph{upper limiting relative measure of $\llb\sigma\rrb$ in $\P$} is
\[
\lambda^+_\P(\sigma)=\limsup_{n\rightarrow\infty}\dfrac{\lambda(\llb\sigma\rrb\cap \llb T_\P\uh n\rrb)}{\lambda(T_\P\uh n)}.
\]

\item The \emph{lower limiting relative measure of $\llb\sigma\rrb$ in $\P$} is
\[
\lambda^-_\P(\sigma)=\liminf_{n\rightarrow\infty}\dfrac{\lambda(\llb\sigma\rrb\cap \llb T_\P\uh n\rrb)}{\lambda(T_\P\uh n)}.
\]
\end{enumerate}
\end{definition}

\noindent For a $\Pi^0_1$ class $\P$ of positive measure, it follows that $\lambda(\sigma\mid\P)=\lambda_\P(\sigma)$ for every $\sigma\in\str$.

The reason that we consider upper and lower limiting relative measure and not merely limiting relative measure is that there are $\Pi^0_1$ classes $\P$ for which the limiting relative measure of some $\sigma\in T_\P$ does not exist.  The following is an example of a decidable $\Pi^0_1$ class $\P$ such that $\lambda_\P(0)$ is not defined.

\begin{example} \label{ex3}  Let $\P$ consist of all  sequences of the form $0 \fr X$ such that $X(4n) = X(4n+1) = 0$ for all $n$ together with all sequences of the form  $1 \fr X$ such that $X(4n+2) = X(4n+3) = 1$ for all $n$. 
Then, for all $n$, we have 
\[
\dfrac{\lambda(\llb 0 \rrb \cap \llb T_\P \uh (4n+1) \rrb) }{\lambda(\llb T_\P \uh (4n+1)  \rrb)} = 1/2 \;\;\;\text{and}\;\;\;
\dfrac{\lambda(\llb 0 \rrb \cap \llb T_\P \uh (4n+3) \rrb)}{\lambda(\llb T_\P \uh (4n+3)  \rrb)} = 1/5. 
\]
Thus $\lambda_\P(0)$ is not defined. 
\end{example}

We use the tree $T_\P$ of extendible nodes to define $\lambda_\P$, as different choices of an underlying tree yield different measures. 
For example, suppose for any tree $T\subseteq\str$ we define 
\[
\lambda_T(\sigma)=\lim_{n\rightarrow\infty}\dfrac{\lambda(\llb\sigma\rrb\cap \llb T\uh n\rrb)}{\lambda(\llb T \uh n \rrb)}.
\]
For the set $\P = \{0^{\omega}, 1^{\omega}\}$, clearly $\lambda_\P(0) = 1/2$.  However, consider the tree 
\[
T = \{0^n: n \in \omega\} \cup \{1^n 0^i: i \leq n\},
\]
so that $\P=[T]$.  Then for all $n$, 
\[
\dfrac{\lambda(\llb 0 \rrb\cap \llb T\uh 2n\rrb)}{\lambda(\llb T \uh 2n \rrb)} = \dfrac{\lambda(\llb 0 \rrb\cap \llb T\uh 2n\rrb)}{\lambda(\llb T \uh 2n \rrb)} = 1/(n+1),
\]
and thus $\lambda_T(0) = 0$.

The measure $\lambda_\P$ is in some sense more natural than the measure $\mu_\P$ considered in Section \ref{Local}.

\begin{example} Let $\P = \{00X\colon X \in \cs\}\cup\{1X\colon X\in\cs\}$. Then $\lambda_\P( 00) = \lambda_\P( 10) = \lambda_\P(11) = \frac13$,
whereas  $\mu_\P(00) = \frac12$ and $\mu_\P(10) = \mu_\P( 11) = \frac14$.
\end{example}

This example indicates the key difference between $\mu_\P$ and $\lambda_\P$:  for each $\sigma\in T_\P$, whereas $\mu_\P$ only takes into account whether $\sigma\fr 0$ and $\sigma\fr1$ are in $T_\P$ in distributing measure to these strings, $\lambda_\P$ takes into account the entire structure of $\P$ above $\sigma$.


Note that the $\Pi^0_1$ class from Example \ref{ex3} above may be seen as the disjoint union of two $\Pi^0_1$ classes $\P_0$ and $\P_1$ where $\lambda_{\P_0}$ and $\lambda_{\P_1}$ are defined but $\lambda_{\P_0 \oplus \P_1}$ is not defined. 
Note that here we have $\# T_{\P_0} \uh 4n = \# T_{\P_0} \uh (4n+2) = 4^n$, whereas $\# T_{\P_1} \uh (4n+2) = \# T_{\P_1} \uh (4n+4) = 4^{n+1}$, so that
\[
\dfrac{\#T_{\P_0} \uh (4n+2)}{\# T_{\P_1} \uh (4n+2)} = 1/4\;\;\;\text{but}\;\;\; \dfrac{\#T_{\P_0} \uh (4n+4)}{\# T_{\P_1} \uh (4n+4)}= 1. 
\]
Thus $\lim_{n \to \infty} \frac{\#T_{\P_0} \uh n }{\# T_{\P_1} \uh n}$ does not exist. 

\begin{proposition} Let $\P = \P_0 \oplus \P_1$ for some $\Pi^0_1$ classes $\P_0$ and $\P_1$ such that $\lambda_{\P_i}(\sigma)$ is defined for each $\sigma$ and for $i=0,1$. Then for each string $\tau$, $\lambda_\P(\tau)$ is defined if and only if $\lim_{n \to \infty} \frac{\#T_{\P_0} \uh n}{ \# T_{\P_1} \uh n}$ exists (including having a limit of infinity). 
\end{proposition}

\begin{proof} Without loss of generality, let $\tau = 0 \fr \sigma$. Then for each $n$, we have
\begin{align*}
 \frac{\lambda(\llb \sigma \rrb\cap \llb T_\P \uh (n+1) \rrb)}{\lambda(\llb T_\P \uh (n+1)  \rrb)} &= \frac12\cdot\frac{ \lambda(\llb \sigma \rrb \cap T_{\P_0} \uh n)}{ (\lambda(T_{\P_0} \uh n) + \lambda(T_{\P_1} \uh n)} 
&= \frac12\cdot\frac{ \lambda(\llb \sigma \rrb \cap T_{\P_0} \uh n)}{\lambda(T_{\P_0} \uh n} \cdot \frac{\lambda(T_{\P_0} \uh n} { (\lambda(T_{\P_0} \uh n) + \lambda(T_{\P_1} \uh n)}.
\end{align*}

For the first fraction above we have 
\[
\lim_{n \to \infty} \frac{ \lambda(\llb \sigma \rrb \cap T_{\P_0} \uh n)}{\lambda(T_{\P_0} \uh n)} = \lambda_{\P_0}(\sigma).
\]

For the second fraction, we have 
\[
 \frac{\lambda(T_{\P_0} \uh n)} { \lambda(T_{\P_0} \uh n) + \lambda(T_{\P_1} \uh n)} = 
 \frac{\# T_{\P_0} \uh n}{ \# T_{\P_0} \uh n + \# T_{\P_1} \uh n } .
\]

Thus if $L = \lim_{n\to \infty}\frac{ \#T_{\P_0} \uh n}{\# T_{\P_1} \uh n}$, then we have 
\[
\lim_{n \to \infty}  \frac{\lambda(T_{\P_0} \uh n)} {\lambda(T_{\P_0} \uh n) + \lambda(T_{\P_1} \uh n)} = 
\lambda_\P(0) =\frac{L}{L+1}.
\]
In particular if $L = \infty$, then this limit will be 1. 
Putting these together, we see that
\[
\lambda_{\P}(0 \fr \sigma) = \frac {\lambda_{\P_0}(\sigma) \cdot L} {2 (L+1)},
\]
where again if $L = \infty$, then $
\lambda_{\P}(0 \fr \sigma) = \frac 12 \lambda_{\P_0}(\sigma)$. 
\end{proof}

It follows from the proof that in fact, for each $\sigma$,  $\lambda_{\P}(0 \fr \sigma)$ exists if and only if $\lambda_{\P_0}(\sigma)$ exists and the limit $L$ exists and similarly for $1 \fr \sigma$. 

\begin{proposition} If $\P_0$ and $\P_1$ are $\Pi^0_1$ classes such that $\lambda_{\P_0}$ and $\lambda_{\P_1}$ are both defined and $\P = \P_0 \otimes \P_1$, then $\lambda_\P$ is also defined.
\end{proposition}

\begin{proof} For each pair $\sigma, \tau$ of strings of the same length $k$ and each $n \geq k$, the number of extensions of $\sigma \oplus \tau$ of length $2n$ in $T_\P$ is exactly the product of the number of extensions of $\sigma$ in $T_{\P_0}$ of length $n$ with the number of extensions of $\tau$ in $T_{\P_1}$ of length $n$.  Applying this fact to the empty string yields 
$\# T_\P \uh 2n = \#T_{\P_0} \uh n \cdot  \# T_{\P_1} \uh n$.  Then we have
\[ 
\frac{\lambda(\llb \sigma \oplus \tau \rrb\cap \llb T_\P \uh 2n\rrb )}{ \lambda(\llb T_\P \uh 2n)} = \frac{\lambda(\llb \sigma \rrb \cap \llb T_{\P_0} \uh n \rrb)} {\lambda(\llb T_{\P_0} \uh n \rrb)} \cdot
\frac{\lambda(\llb \sigma \rrb \cap \llb T_{\P_1} \uh n \rrb)} {\lambda(\llb T_{\P_1} \uh n \rrb)},
\]
and therefore  $\lambda_\P(\sigma \oplus \tau) = \lambda_{\P_0}(\sigma) \cdot \lambda_{\P_1}(\tau)$. Since intervals of the from $\llb \sigma \oplus \tau \rrb$ form a basis, it follows that $\lambda_\P$ is defined for all intervals. 
\end{proof}

Later we will consider some notions of homogeneity which will provide conditions on the class $\P$ under which $\lambda_{\P}$ is defined.

\subsection{Randomness with respect to the conditional measure $\lambda_\P$}

We now turn to the global definition of randomness in a $\Pi^0_1$ class $\P$.

\begin{definition}
$X$ is \emph{globally random in $\P$} if $X\in\MLR_{\lambda_\P}^{T_\P}$.  
\end{definition}

By the relativized Levin-Schnorr theoreom, $X$ is globally random in $\P$ if and only if 
\[
K^{T_\P}(X\uh n)\geq -\log\lambda_\P(X\uh n)-O(1)
\]
for all $n$.  We will make use of this characterization of global randomness in $\P$ in the ensuing discussion.

The next two examples show that these measures give rise to different notions of algorithmic randomness.  

\begin{example}  Let $\Q = \{0^{\omega}\} \cup \{1 X: X \in \cs\}$.  Then $\lambda_\Q(0^{n+1}) = \frac 1{2^n+1}$, so that $0^{\omega}$ is not random with respect to $\lambda_\Q$ random.  However, $\mu_\Q(0^{n+1}) = \frac12$, for every $n$, so that $0^{\omega}$ is random with respect to $\mu_\Q$. That is, if $0^{\omega}$ belongs to an open set $\U$, then some $n$ such that $\llb 0^{n+1} \rrb \subseteq U$, so that $\mu_\Q(\U) = \frac{1}{2}$. Hence $0^{\omega}$ must pass every $\mu_\P$-Martin-L\"of test. 
\end{example}

\begin{example}\label{ex6}  Let $\P = \{0^n1^{\omega}: n \in \omega\}\cup\{0^\omega\}$. We observe that for each $m$, $T_\P \uh m$ has exactly $m+1$ elements, namely, $0^m, 0^{m-1} 1, 0^{m-2}11, \dots, 1^m$.   Now fixing $k$, we see that  $T_\P \uh (k+n)$ has exactly $n+k+1$ elements, of which $n+1$ extend $0^{k}$.  It follows that 
\[ 
\frac{ \lambda(\llb 0^k \rrb\cap \llb T\uh (k+n)\rrb)}{\lambda(\llb T \uh (k+n) \rrb)} = \frac {n+1}{n+k+1},
\]
so that $\lambda_\P( 0^k ) = 1$ for all $k\in\omega$. Thus $0^{\omega}$ is random with respect to $\lambda_\P$.  On the other hand, we see that $\mu_\P(\{0^{n} 1^{\omega}\}) = 2^{-k-1}$ for each $k$ and that $\mu_\P(\{0^{\omega}\}) = 0$. 
Thus $\mu_\P( 0^{k}) = 2^{-k}$ for every $k\in\omega$, so that $0^{\omega}$ is not $\mu_\P$-Martin-L\"of random. 
\end{example}

Example \ref{ex6} shows that there are $\Pi^0_1$ classes $\P$ where the collection of $\lambda_\P$-random sequences is disjoint from the collection of $\mu_\P$-random sequences.  An example of such a class without isolated points is the following.

\begin{example} \label{ex7}  Let $T$ be the tree obtained by closing the set $\{00,01,11\}^{<\omega}$ downward under $\preceq$ and let $\P=[T]$.  Note that $\P$ is homeomorphic to $3^\omega$.  Now, the measure $\mu_\P$ on $\P$ is defined by
\begin{itemize}
\item[] (i) $\mu_\P(\sigma00)=\mu_\P(\sigma01)=\frac{1}{4}\mu_\P(\sigma)$, and
\item[] (ii) $\mu_\P(\sigma11)=\frac{1}{2}\mu(\sigma)$
\end{itemize}
for every $\sigma\in \{00,01,11\}^{<\omega}$.  In addition, the measure $\lambda_\P$ on $\P$ satisfies
\begin{itemize}
\item[] (iii) $\lambda_\P(\sigma00)=\lambda_\P(\sigma01)=\lambda_\P(\sigma11)=\frac{1}{3}\lambda_\P(\sigma)$
\end{itemize}
for every $\sigma\in \{00,01,11\}^{<\omega}$.

Let $\nu_0$ be the Bernoulli ($\frac{1}{4},\frac{1}{4},\frac{1}{2}$)-measure on $3^\omega$ and let $\nu_1$ be the Bernoulli ($\frac{1}{3},\frac{1}{3},\frac{1}{3}$)-measure on $3^\omega$.  For $i=0,1$, every $\nu_i$-random sequence satisfies the law of large numbers with respect to $\nu_i$, which implies in particular that
\[
\lim_{n\rightarrow\infty}\frac{\#_2(Z\uh n)}{n}=\frac{1}{3}
\]
for every $\nu_0$-random sequence $Z\in 3^\omega$ and
\[
\lim_{n\rightarrow\infty}\frac{\#_2(Z\uh n)}{n}=\frac{1}{2}
\]
for every $\nu_1$-random sequence $Z\in 3^\omega$.

Let $\Phi:\P\rightarrow 3^\omega$ be the Turing functional induced by the map from $\{00,01,11\}$ to $\{0,1,2\}$ given by:
\begin{center}
$00\mapsto 0$\\
$01\mapsto 1$\\
$00\mapsto 2$
\end{center}
Then $\Phi$ is a measure-preserving isomorphism between $(\P,\lambda_\P)$ and $(3^\omega,\nu_0)$ as well as between $(\P,\mu_\P)$ and $(3^\omega,\nu_1)$.

Suppose that there is some $X\in\P$ that is random with respect to both $\lambda_\P$ and $\mu_\P$.  Then setting $Y=\Phi(X)$, by the preservation of randomness, it follows that $Y$ is random with respect to both $\nu_0$ and $\nu_1$, which is impossible by our above remarks.
\end{example}

\section{An intermediate approach to defining randomness in a $\Pi^0_1$ class} \label{sec-Intermediate}

Let us now consider the notion of initial segment complexity for members of a given $\Pi^0_1$ class $\P$ briefly considered in Subsection \ref{subsec-isc1}.

\begin{definition} \label{def:gi}
$X\in\P$ is \emph{$T_\P$-incompressible} if
\[
K^{T_\P}(X\uh n)\geq \log\#T_\P\uh n-O(1).
\]
\end{definition}

For the previous two initial segment complexity notions we have considered, the complexity threshold for a given $\sigma\in T_\P$ is determined by (i) the branching in $\P$ along initial segments of $\sigma$ and (ii) the branching in $\P$ along extensions of $\sigma$, respectively.  For $T_\P$-incompressible sequences, the complexity threshold is now determined by the branching 
along not just initial segments of $\sigma$ but along initial segments of \emph{all} strings in $T_\P$ of length $|\sigma|$.

Recall that a sequence $X$ is complex if there is a computable, nondecreasing, unbounded function $f:\omega\rightarrow\omega$ such that
$K(X\uh n)\geq f(n)$ for all $n\in\omega$. In \cite[Theorem 2.13]{Bin08}, Binns proved that if $\P$ is a $\Pi^0_1$ class containing a complex element, then for every $Y\in\cs$, there is some $X\in\P$ such that $Y\leq_T X$ (in fact, he proved a stronger result).  In particular, such a class must be uncountable.  Using these facts, we can show: 

\begin{proposition}
If $\P$ is a countable, decidable $\Pi^0_1$ class, then no $X\in\P$ satisfies
\[
(\exists c)(\forall n)K^{T_\P}(X\uh n)\geq \log\# T_\P\uh n-c.
\]
\end{proposition}

\begin{proof}
Since $\P$ is decidable, then $T_\P$ is computable and thus each $T_\P$-incompressible sequence $X\in\P$ satisfies
\[
K(X\uh n)\geq \log\#T_\P\uh n-O(1).
\]
Since for $c\in\omega$, the function $n\mapsto (\log\#T_\P\uh n)-c$ is computable, nondecreasing, and unbounded, it follows that such an $X$ is complex.  By the above discussion, if $\P$ contains any $T_\P$-incompressible sequence, then it contains a complex member and thus is uncountable.  This contradicts the fact that $\P$ is countable, and thus $\P$ contains no $T_\P$-incompressible sequences.
\end{proof}
%

%

The following is open:
\begin{question}
Is there an undecidable $\Pi^0_1$ class $\P$ such that no $X\in\P$ is $T_\P$-incompressible?
\end{question}


As the $\Pi^0_1$ class $\P$ from Example \ref{ex6} is countable and decidable, it follows that no $X\in\P$ is $T_\P$-incompressible.  We saw that in that case, the collection of $\lambda_\P$-random sequences in $\P$ is disjoint from the collection of $\mu_\P$-random sequences in $\P$.  It thus follows that neither $\lambda_\P$-randomness nor $\mu_\P$-randomness imply $T_\P$-incompressibility.  As we will see in the next section, being $T_\P$-incompressible does not imply being a randomly produced path for some $\Pi^0_1$ classes $\P$ (see the remarks after Corollary \ref{cor-a}).  Moreover, Example \ref{ex3} shows that $T_\P$-incompressibility does not imply globally randomness for all $\Pi^0_1$ 
classes $\P$.

To see this, letting $\P$ be the $\Pi^0_1$ class from Example \ref{ex3}, one can calculate that, for $k\in\omega$, 
\[
\# T_\P\uh(4k+1)=2^{2k+1}
\]
 and hence $\log\# T_\P\uh(4k+1)=2k+1$.
One can further calculate that for each $k$ and each $\sigma\in T_\P\uh(4k+1)$, 
\[
\theta_\P(\sigma)=2k+1=\log\#T_\P\uh |\sigma|.
\]

Lastly, for each $k\in\omega$ and each $\sigma\in T_\P\uh(4k+j)$ for $j=2,3,4$, we have
\[
|\theta_\P(\sigma)-\theta_\P(\sigma\uh (4k+1))|\leq 3
\]
and
\[
|\log\#T_\P\uh (4k+j)-\log\#T_\P\uh (4k+1)|\leq 3.
\]
Thus, for all $\sigma\in T_\P$, 
\[
\bigl|\theta_\P(\sigma)-\log\# T_\P\uh |\sigma|\bigr|\leq O(1).
\]
From this it follows that $X\in\P$ is a randomly produced path through $T_\P$ if and only if $X$ is $T_\P$-incompressible.  As $\P$ does not contain any globally random sequences, as $\lambda_\P$ is not defined, it follows that $T_\P$-incompressibility does not imply global randomness in $\P$.

%
%

Significantly, we are not aware of precisely which conditions guarantee that $T_\P$-incompressible sequences exist for a given $\Pi^0_1$ class $\P$.  However, in the next section we will identify several sufficient conditions for a $\Pi^0_1$ class to contain a $T_\P$-incompressible member.

\section{Notions of homogeneity}\label{sec-Homogeneous}

In this section, we explore various notions of homogeneity for $\Pi^0_1$ classes, with the aim of showing that all of the notions of randomness for members of $\Pi^0_1$ classes coincide for sufficiently homogeneous $\Pi^0_1$ classes.
In the next three subsections, we will first consider notions of homogeneity that have been considered in the computability-theoretic literature.  In Subsection \ref{sec6.4}, we will introduce a new notion of homogeneity, namely \emph{additive homogeneity}.

\subsection{Separating set homogeneity}  
 The first notion of homogeneity we will consider is a classic one, playing a useful role in the study of c.e.\ sets and the incompleteness phenomenon.

\begin{definition}
A $\Pi^0_1$ class $\P$ is \emph{homogeneous} if for every $\sigma, \tau\in T_\P$ such that $|\sigma|=|\tau|$, $\sigma\fr i\in T_\P$ if and only if $\tau\fr i\in T_\P$ for $i\in\{0,1\}$.
\end{definition}

Recall that a $\Pi^0_1$ class $\P$ is homogeneous if and only if there are disjoint c.e.\ sets $A$ and $B$ such that every $C\in\P$ is a separating set for $A$ and $B$, i.e., $A\subseteq C$ and $B\cap C=\emptyset$.  In this case, we write $\P=S(A,B)$.  As we will later introduce alternative notions of homogeneity for $\Pi^0_1$ classes, we will hereafter refer to homogeneity in the above sense as \emph{s.s.-homogeneity} (for \emph{separating set homogeneity}).

\begin{lemma} \label{lem-sshom1}
If $\P$ is an s.s. homogenous $\Pi^0_1$ class, then $\# T_\P\uh |\sigma|=2^{\theta_\P(\sigma)}$ for every $\sigma\in T_\P$.
\end{lemma}

\begin{proof} For the root node $\epsilon$, we have $\theta_\P(\epsilon) = 0$ and $\#T_\P \uh 0= 1$.  
Now suppose the claim holds for all strings of length $n$.  Given $\tau\in T_\P$ of length $n+1$, then $\tau=\sigma\fr i$ for some $\sigma\in T_\P$ of length $n$ and some $i\in\{0,1\}$.  We have two cases to consider.\\

\noindent \emph{Case 1:} If $\tau^\curvearrowright=\sigma\fr(1-i)\notin T_\P$, we have $\theta_\P(\tau)=\theta_\P(\sigma)$.  But since $\P$ is homogeneous, for every $\rho\in T_\P$ of length $n$, we must have $\rho\fr i\in T_\P$ and $\rho\fr(1-i)\notin T_\P$.  Thus $\#T_\P\uh |\tau|=\# T_\P\uh |\sigma|=2^{\theta_\P(\sigma)}=2^{\theta_\P(\tau)}$.\\

\noindent \emph{Case 2:} If $\tau^\curvearrowright=\sigma\fr(1-i)\in T_\P$, we have $\theta_\P(\tau)=\theta_\P(\sigma)+1$.  But since $\P$ is homogeneous, for every $\rho\in T_\P$ of length $n$, we must have $\rho\fr i\in T_\P$ for $i=0,1$.  Thus $\#T_\P\uh |\tau|=2\cdot \# T_\P\uh |\sigma|=2^{\theta_\P(\sigma)+1}=2^{\theta_\P(\tau)}$.\\

\noindent Hence by induction, the conclusion holds.
\end{proof}

\begin{lemma}\label{lem-sshom2}
If $\P$ is an s.s.  homogenous $\Pi^0_1$ class, then $\lambda_\P(\sigma)=2^{-\theta_\P(\sigma)}$ for every~$\sigma\in T_\P$.
\end{lemma}

\begin{proof}

For $n\geq |\sigma|$, for each $\sigma\in T_\P\uh n$, we consider
\[
\dfrac{\lambda(\llb\sigma\rrb\cap \llb T_\P\uh n\rrb)}{\lambda(T_\P\uh n)}.
\]
Suppose that $\sigma$ has $j$ extensions in $T_\P$ of length $n$.  Since $\P$ is homogenous, we must have $j=2^k$ for some $k\leq n-|\sigma|$.  Thus 
\begin{equation}\label{eq1}
\lambda(\llb\sigma\rrb\cap\llb T_\P\uh n\rrb)=2^{-n}2^k.
\end{equation}
 Moreover, by the homogeneity of $\P$, $\# T_\P\uh n=2^k\cdot\# T_\P\uh|\sigma|,$ so that 
\begin{equation}\label{eq2}
 \lambda(T_\P\uh n)=2^{-n}\cdot\# T_\P\uh n=2^{-n}2^k\cdot\# T_\P\uh|\sigma|.
\end{equation}
Combining (\ref{eq1}) and (\ref{eq2}) with the fact that $\# T_\P\uh|\sigma|=2^{-\theta_\P(\sigma)}$ yields
\[
\dfrac{\lambda(\llb\sigma\rrb\cap \llb T_\P\uh n\rrb)}{\lambda(T_\P\uh n)}=\dfrac{2^{-n}2^k}{2^{-n}2^k\cdot\# T_\P\uh|\sigma|}=2^{-\theta_\P(\sigma)}.
\]
Thus
\[
\lambda_\P(\sigma)=\lim_{n\rightarrow\infty}\dfrac{\lambda(\llb\sigma\rrb\cap \llb T_\P\uh n\rrb)}{\lambda(T_\P\uh n)}=\lim_{n\rightarrow\infty}2^{-\theta_\P(\sigma)}=2^{-\theta_\P(\sigma)}.
\]

\end{proof}

For an s.s.-homogeneous $\Pi^0_1$ class $\P$ and $X\in\P$, by Lemmas \ref{lem-sshom1} and \ref{lem-sshom2}, we have
\[
\log\#T_\P\uh n=\theta_\P(X\uh n)=-\log\lambda_\P(X\uh n)
\]
for each $n\in\omega$.  As an immediate consequence we have:

\begin{theorem} \label{thm-sshom}
 Let $\P$ be an s.s.-homogeneous $\Pi^0_1$ class.  The following are equivalent for $X\in\P$:
\begin{itemize}
\item[(i)] $X$ is a random path through $\P$.
\item[(ii)] $X$ is $T_\P$-Martin-L\"of random with respect to $\lambda_\P$.
\item[(iii)] $X$ is $T_\P$-incompressible.
\end{itemize}
\end{theorem}

Finally, we observe that by Lemmas \ref{lem-sshom1} and \ref{lem-sshom2}, any s.s.-homogeneous $\Pi^0_1$ class $\P$ and any $\sigma \in T_\P\uh n$, we have $\lambda_\P(\sigma) = 1/\# T_\P \uh n$. Moreover for any clopen set $\U = \llb \sigma_1 \rrb \cup \dotsb \cup \llb \sigma_m \rrb$ such that $\sigma_i\in T_\P$ and $|\sigma_i|=n$ for $1\leq i\leq m$, we have
\[
\lambda_\P(\U) =\sum_{i=1}^k=m\cdot\lambda_\P(\sigma_i)=\frac{m}{\#T_\P\uh n}=\frac{\#T_\U\uh n}{\#T_\P\uh n}.
\]
From this it follows that for any s.s.-homeogeneous $\Pi^0_1$ class $\P$ and any $\Pi^0_1$ $\Q \subseteq \P$, we have $\Q = \bigcap_{n\in\omega} \llb T_\Q \uh n \rrb$, so that 

\[
\lambda_\P(\Q) = \lim_{n \to \infty} \lambda_\P(\llb T_\Q \uh n \rrb) = \lim_{n \to \infty} \frac{\# T_\Q \uh n}{ \#T_\P \uh n}.
\]
This latter limit exists since the sequence $(\frac{\# T_\Q \uh n}{ \#T_\P \uh n})_{n\in\omega}$ is nondecreasing and bounded from below by 0.

\subsection{$n$-homogeneity and weakly $n$-homogeneity}

Generalized c.e. separating classes and their degrees of difficulty were investigated by Cenzer and Hinman \cite{CH08}.  For example, given four disjoint c.e. sets $A_0,A_1,A_2,A_3$, consider the $\pz$ class  $\P \subseteq \{0,1,2,3\}^{\omega}$ such that $X \in \P \iff (\forall m)(\forall i \leq 3)(X(m) = i \to m \notin A_i)$.  Now the class $\P$ can be represented by a $\pz$ class $\Q$ in $\{0,1\}^{\omega}$ by letting $X \in \{0,1\}^{\omega}$ represent $Y \in \{0,1,2,3\}^{\omega}$ by replacing $Y(m)$ with  $X(2m) X(2m+1)$,  where each 0 in $Y$ is replaced by $00$, each $1$ by $01$,  each 2 by $10$, and each 3  by $11$.  Such a class $\Q$ will have the property that, for any two nodes $\sigma$ and $\tau$ in $T_\Q$ of length $2m$ and any string $\rho$ of length $2$, $\sigma \fr \rho \in T_\Q$ if and only if  $\tau \fr \rho \in T_\Q$.  We will give a more general definition here. 

\begin{definition}
Let $n\in\omega$.  Then a $\Pi^0_1$ class $\P$ is \emph{$n$-homogeneous} if for every $k\in\omega$, every $\sigma_0,\sigma_1\in T_\P\uh nk$, and every $\tau \in \{0,1\}^n$,
$\sigma_0 \fr \tau \in T_\P$ if and only if $\sigma_1 \fr \tau \in T_\P$. 
$\P$ is \emph{weakly $n$-homogeneous} if, for each $k$ and each $\sigma_0$ and $\sigma_1$ in $T_\P\uh nk$,
\[
\#T_{\P\cap\llb\sigma_0\rrb}\uh n(k+1)=\#T_{\P\cap\llb\sigma_1\rrb}\uh n(k+1).
\] 
\end{definition}
It is easy to see that if $\P$ is $n$-homogeneous, then $\P$ is weakly $n$-homogeneous. Note further that 1-homogeneity is simply s.s.-homogeneity.

Returning briefly to the notions of product and disjoint union, we note that the disjoint union of $n$-homogeneous classes need not be $n$-homogeneous,
as seen by Example \ref{ex3}. For products, the following are easy to see. 

\begin{proposition} 
\begin{itemize}
\item[(i)] If $\P$ and $\Q$ are both $n$-homogeneous, then $\P \oplus \Q$ is $2n$-homogeneous.
\item[(ii)] If $\P$ and $\Q$ are both weakly $n$-homogeneous, then $\P \oplus \Q$ is weakly $2n$-homogeneous. 
\end{itemize}
\end{proposition}

If $\P$ is weakly $m$-homogeneous and $\Q$ is weakly $n$-homogeneous for $m,n\in\omega$, then $\P \oplus \Q$ is \emph{almost} weakly $mn$-homogeneous, in the following sense:  For any strings $\tau_0$ and $\tau_1$ in $T_{\P \oplus \Q}$ of length $mnk+1$, $\tau_0$ and $\tau_1$ will have an equal number of extensions 
of length $mn(k+1)+1$.  The details are left to the reader. 
  
\begin{lemma} If $\P$ is $n$-homogeneous, then $\lambda_\P(\sigma)$ is defined for every $\sigma\in\str$. 
\end{lemma}

\begin{proof} Since $\P$ is $n$-homogeneous, it follows that, for each $k$ and each $\sigma \in T_\P \cap \{0,1\}^{nk}$, 
\[
\lambda_\P(\sigma) = \frac{1}{\#T_{\P \cap \llb \sigma \rrb} \uh nk}.
\] Now suppose $|\sigma| = nk +i$ where $0 < i < n$, and let $m$ be the number of extensions in $T_\P$ of $\sigma$ 
of length $n(k+1)$.  Then 
\[
\lambda_\P(\sigma) = \frac{m}{\#T_{\P\cap \llb \sigma \rrb} \uh n(k+1)}.
\] 
\end{proof}

Note that by Example \ref{ex3}, $\lambda_\P$ is not necessarily defined for weakly $n$-homogeneous classes.  In the case that $\P$ is weakly $n$-homogeneous and $\lambda_\P$ is defined everywhere, we have the following.

\begin{theorem} Let $\P$ be a weakly $n$-homogeneous $\pz$ class $\P$ for some $n\in\omega$ and suppose that $\lambda_\P$ is defined.
 Then for any $X \in \P$, $X$ is globally random in $\P$ if and only if $X$ is $T_\P$-incompressible.
\end{theorem} 

\begin{proof} Suppose that $\P$ is weakly $n$-homogeneous. Then there exists a sequence $c_1,c_2,\dotsc$ of positive integers such that, for each $k$ and each $\sigma \in T_\P \cap  \{0,1\}^{nk}$, $\sigma$ has exactly $c_k$ extensions in $T_\P$ of length $n(k+1)$. Fix $j$ and  a string $\sigma \in T_\P \cap  \{0,1\}^{nj}$.  Then for any $k > j$,
$T_\P \uh kn$ has exactly $c_1 c_2 \cdots c_k$ elements, of which $c_{j+1} c_{j+2} \cdots c_k$ are extensions of $\sigma$.  It follows that 
\[
\dfrac{\lambda(\llb\sigma\rrb\cap \llb T_\P\uh nk \rrb)}{\lambda(T_\P\uh nk)} = \frac{1}{c_1 c_2 \cdots c_j}. 
\]
Given that $\lambda_\P(\sigma)$ exists, this means that $\lambda_\P(\sigma) = \frac{1}{c_1 c_2 \cdots c_j}$. 
Thus $\lambda_\P(\sigma) = \frac{1}{\#T_\P \uh |\sigma|}$.

Next suppose that $nj < |\sigma|  < n(j+1)$.  It is clear that 
\begin{equation}\label{eq-weakhom1}
c_1 \cdots c_j \leq \#T_\P \uh |\sigma| \leq c_1 \cdots c_{j+1}.
\end{equation}
Now observe that $\sigma$ has at most $c_{j+1}$ extensions in $T_\P$ of length $n(j+1)$. 
It follows that for all $k > j$, 
\[
\frac{1}{c_1 \cdots c_{j+1}}  \leq \dfrac{\lambda(\llb\sigma\rrb\cap \llb T_\P\uh nk \rrb)}{\lambda(T_\P\uh nk)} \leq  \frac{1}{c_1 c_2 \cdots c_j}.
\]
Thus, given that $\lambda_\P(\sigma)$ exists, we have 
\begin{equation}\label{eq-weakhom2}
\frac{1}{c_1 \cdots c_{j+1}} \leq \lambda_\P(\sigma) \leq  \frac{1}{c_1 c_2 \cdots c_j}.
\end{equation}
Putting inequalities (\ref{eq-weakhom1}) and (\ref{eq-weakhom2}) together, we have
\[
\frac{\# T_\P \uh |\sigma|}{c_{j+1}} \leq \frac{1}{\lambda_\P(\sigma)} \leq c_{j+1}\cdot \#T_\P \uh |\sigma|.
\]

Noting that each $c_{j+1} \leq 2^{n+1}$, and taking logarithms, we obtain
\[
\log \# T_\P \uh |\sigma| -(n+1) \leq - \log \lambda_\P (\sigma) \leq \log \#T_\P \uh |\sigma|+(n+1).
\]
Since $n$ is a fixed constant, it follows from the Levin-Schnorr Theorem \ref{thm-levinschnorr} and Definition \ref{def:gi}
that an element $X$ of $\P$  is globally random in $\P$ if and only if $X$ is $T_\P$- incompressible. 
\end{proof}

\begin{corollary}  \label{cor-a}
For any $n$-homogeneous $\pz$ class $\P$, and any element $X$ of $\P$, $X$ is globally random in $\P$ if and only 
$X$ is $T_\P$-incompressible. 
\end{corollary}

%

Note that the $\Pi^0_1$ class from Example \ref{ex7} is 2-homogeneous and thus weakly 2-homogeneous, from which it follows that the randomly produced paths in a weakly $2$-homogeneous class $\P$ need not coincide with the $T_\P$-incompressible members of $\P$ (in fact, as in Example \ref{ex7}, these two classes many even be disjoint).  By the comments following Example \ref{ex7}, we can extend this observation to any $n$-homogeneous class for $n\geq 2$.

\subsection{Van Lambalgen's notions of multiplicative homogeneity}

In an attempt to find the broadest notion of homogeneity for $\Pi^0_1$ classes for which the corresponding analogue of Theorem \ref{thm-sshom} holds, we will next consider a notion of homogeneity due to van Lambalgen \cite{Van87}. The idea behind this notion is that a $\Pi^0_1$ class $\P$ is homogeneous if the amount of branching in the subclasses of $\P$ does not differ too much from the amount of branching in the class $\P$ as a whole.    

\begin{definition}[Van Lambalgen \cite{Van87}]  Let $\P$ be a $\Pi^0_1$ class.  Then $\P$ is \emph{VL-homogeneous} if there is some constant $c\in\omega$ such that the following two conditions are satisfied:
\begin{itemize}
\item[(i)] for every $\Pi^0_1$ subclass $\Q\subseteq\P$, for every $n$ and every $k\geq n$,
\[ \tag{$\dagger$}
\dfrac{\# T_\Q\uh k}{\# T_\Q\uh n}\leq c\cdot\dfrac{\# T_\P\uh k}{\# T_\P\uh n};
\]
\item[(ii)] for every $\sigma\in T_\P$, if we set $\Q=\llb\sigma\rrb\cap\P$, we have, for every $k\geq n\geq |\sigma|$,
\[\tag{$\dagger\dagger$}
\dfrac{\# T_\P\uh k}{\# T_\P\uh n}\leq c\cdot\dfrac{\# T_\Q\uh k}{\# T_\Q\uh n}.
\]
\end{itemize}
\end{definition}

Given that we require the amount of branching in subclasses of $\P$ to be within a multiplicative constant of the amount of branching in $\P$, we refer to these notions as notions of \emph{multiplicative} homogeneity.   We now prove two lemmas that simplify the verification that a given $\Pi^0_1$ class is VL-homogeneous.

\begin{lemma} \label{lem-vlhom1}
Suppose that condition $(\dagger)$ holds for a fixed $c\in\omega$ and every clopen subset $\U$ of $\P$. Then $(\dagger)$ holds for any closed set $\Q$ of $\P$.
\end{lemma}

\begin{proof} Given $\Q$, $n$ and $k \geq n$, let $\U = \llb T_\Q \uh k \rrb$.  Then $T_\Q \uh k = T_\U \uh k$ and also $T_\Q \uh n = T_\U \uh n$. It follows that 
\[
\dfrac{\# T_\U \uh k}{\# T_\P \uh k} = \dfrac{\# T_\Q \uh k}{\# T_\P\uh k} \leq c\cdot\dfrac{\# T_\Q \uh n}{\# T_\P\uh n} =  c\cdot\dfrac{\# T_\U \uh n}{\# T_\P \uh n},
\]
from which $(\dagger)$ follows.
\end{proof}

\begin{lemma} \label{lem-vlhom2}
Suppose there is a fixed $c\in\omega$ such that condition $(\dagger)$ holds for any $\pz$ $\Q$ and any $n,k\in\omega$ that together satisfy $\#T_\Q \uh n = 1$ and $k \geq n$. Then condition $(\dagger)$ holds for all closed sets $\Q \subseteq \P$. 
\end{lemma}

\begin{proof} By Lemma \ref{lem-vlhom1}, it suffices to show that condition $(\dagger)$ holds for all clopen $\U\subseteq\P$.  Fixing $n$ and $k$, let $T_\U \uh n = \{\sigma_1, \dots, \sigma_m\}$ and let $\U_i = \llb \sigma_i \rrb \cap \U$
and $T_i = T_{\U_i}$ for $i = 1,\dots,m$. Since $T_i \uh n$ and $T_j \uh n$ are disjoint for $i \neq j$, it follows that  $T_i \uh k$ and $T_j \uh k$ are also disjoint for $i \neq j$. Then by assumption, for each $i$, 
\[
\dfrac{\# T_i \uh k}{\# T_\P\uh k}\leq c \cdot \dfrac{\# T_i \uh n}{\# T_\P\uh n}.
\]
It follows that
\[
\dfrac{\# T_\U \uh k}{\# T_\P \uh k} = \dfrac{\sum_{i = 1}^m \#T_i \uh k}{\# T_\P \uh k}      \leq c \cdot \dfrac{\sum_{i = 1}^m \#T_i \uh n}{\# T_\P \uh n}  = c \cdot
\dfrac{\# T_\U \uh n}{\# T_\P\uh n}.
\]
\end{proof} 

As observed by van Lambalgen, the satisfaction of the condition ($\dagger$) for a $\Pi^0_1$ class $\P$ does not rule out the possibility that $\P$ has an isolated point, as ($\dagger$) only guarantees that the branching in any $\Pi^0_1$ subclass of $\P$ does not exceed that in $\P$ (up to a multiplicative constant).  The additional condition ($\dagger\dagger$) thus guarantees that VL-homogeneous classes do not contain isolated points.

Clearly $\cs$ itself is VL-homogeneous.  We obtain many more examples from the following:

\begin{theorem} \label{thm-weaktovl} For any $\Pi^0_1$ class $\P$ and any $n\in\omega$, if $\P$ is weakly $n$-homogeneous, then $\P$ is VL-homogeneous. 
\end{theorem}

\begin{proof} Let $\P$ be weakly $n$-homogeneous and let $c_1,c_2,\dots$ be given so that for any string $\tau$ of length $mk$ in $T_\P$,
$\tau$ has exactly $c_k$ extensions in $T_\P$ of length $m(k+1)$. Let $\Q \subseteq \P$ and let $i \in \omega$ and $s,t < m$.

Then we have the following inequalities:

\[
c_k c_{k+1} \dotsb c_{k+i-1} \leq  \frac{\# T_{\P} \uh (m(k+i) + t)}{\# T_{\P} \uh (mk+s)}  \leq 2^{m-s} c_k c_{k+1} \dotsb c_{k+i-1} 2^t \leq 2^{2m} c_k c_{k+1} \dotsb c_{k+i-1}
\]
and
\[
\frac{\# T_{\Q} \uh (m(k+i) + t)}{\# T_{\Q} \uh (mk+s)}  \leq 2^{m-s} c_k c_{k+1} \dotsb c_{k+i-1} 2^t.
\]
It follows that
\[
\frac{\# T_{\Q} \uh (m(k+i) + t)}{\# T_{\Q} \uh (mk+s)} \leq 2^{2m} \cdot \frac{\# T_{\P} \uh (m(k+i) + t)}{\# T_{\P} \uh (mk+s)}.
\]
Thus the condition $(\dagger)$ is satisfied with constant $c = 2^{2m}$, so that  $\P$ is VL-homogeneous.

Now suppose that $\Q = \llb \sigma \rrb \cap \P$ for some string $\sigma \in T_\P$ and let $k$ satisfy $|\sigma| \leq mk$. Then for $i,s,t$ as above,
\[
c_k c_{k+1} \dotsb c_{k+i-1} \leq  \frac{\# T_{\Q} \uh (m(k+i) + t)}{\# T_{\Q} \uh mk+s)},
\]
so that by the first inequality above, 
\[
\frac{\# T_{\P} \uh (m(k+i) + t)}{\# T_{\P} \uh (mk+s)} \leq  2^{2m} \frac{\# T_{\Q} \uh (m(k+i) + t)}{\# T_{\Q} \uh (mk+s)}.
\]
This verifies that the condition $(\dagger \dagger)$ will be satisfied by constant $d =  2^{2m}$, for all $n \geq mk$, so that $\P$ is VL-homogeneous. 
\end{proof}

By the comments following Corollary \ref{cor-a}, the $\Pi^0_1$ $\P$ class from Example \ref{ex7} is weakly $n$-homogeneous and thus VL-homogeneous by the above result.  Since the notions of randomness from Sections \ref{Local} - \ref{sec-Intermediate} do not coincide in $\P$, it follows that VL-homogeneity is also not sufficient to guarantee the equivalence of these notions.

Recalling that for an s.s.-homogeneous $\Pi^0_1$ class $\P$ and any $\Pi^0_1$ $\Q \subseteq \P$, we have $\lambda_\P(\Q) = \lim_{n \to \infty} \frac{ \# T_\Q \uh n}{\# T_\P \uh n}$ as  discussed previously, we would like see what conditions ensure that
\[
\lim_{n\rightarrow\infty}\dfrac {\# T_\Q \uh n}{\#T_\P \uh n}
\]
 exists for subclasses $\Q$ of a $\pz$ class $\P$.  Recall Example \ref{ex3} above, for which 
 \[
 \lim_{n \to \infty} \frac {\lambda(\llb 0 \rrb \cap \llb T_\P \uh n\rrb)}{\lambda(\llb T_\P \uh n\rrb)}
 \]
  does not exist.  This class is weakly 4-homogeneous and therefore is vL-homogeneous by Theorem \ref{thm-weaktovl}.
Thus VL-homogeneity does not suffice to ensure the existence of 
the limit 
\[
\lim_{n \to \infty} \frac{ \# T_{\P\cap\llb 0\rrb} \uh n}{\# T_\P \uh n}
\]
and consequently it does not suffice to ensure that $\lambda_\P$ is defined everywhere. Note also that since there are random paths through $\P$ that are not globally random, $\P$ provides another example of a VL-homogeneous class in which the various notions of randomness for $\Pi^0_1$ classes fail to coincide.

%
%

In the case that $\lambda_\P$ is defined in a VL-homogeneous $\Pi^0_1$ class $\P$, we can show the coincidence of global randomness and $T_\P$-incompressibility.  First we prove a lemma.

\begin{lemma}\label{lem-upperlower}
Let $\P$ be a VL-homogeneous $\Pi^0_1$ class. 
\begin{itemize}
\item[(i)] If $\lambda_\P(\sigma)$ is defined, then
\[
\Bigl|-\log\lambda_\P(\sigma)-\log\#T_\P\uh|\sigma|\Bigr|\leq O(1).
\]
\item[(ii)] If $\lambda_\P(\sigma)$ is not defined, we still have
\[
\Bigl|-\log\lambda_\P^+(\sigma)-\log\#T_\P\uh|\sigma|\Bigr|\leq O(1)
\]
and
\[
\Bigl|-\log\lambda_\P^-(\sigma)-\log\#T_\P\uh|\sigma|\Bigr|\leq O(1).
\]

\end{itemize}
\end{lemma}

\begin{proof}
Let $\sigma\in\str$.  Let $k\geq |\sigma|$.  Note that $\llb\sigma\rrb\cap\llb T_\P\uh k\rrb=\llb T_{\P\cap\llb\sigma\rrb}\uh k\rrb$, so that 
\[
\lambda(\llb\sigma\rrb\cap\llb T_\P\uh k\rrb)=\lambda(T_{\P\cap\llb\sigma\rrb}\uh k)=2^{-k}\cdot \# T_{\P\cap\llb\sigma\rrb}\uh k.
\]
Thus we can write
\begin{equation}\label{eq-another}
\frac{\lambda(\llb\sigma\rrb\cap\llb T_\P\uh k\rrb)}{\lambda(T_\P\uh k)}=\frac{2^{-k}\cdot \# T_{\P\cap\llb\sigma\rrb}\uh k}{2^{-k}\cdot\#T_\P\uh k}=\frac{\# T_{\P\cap\llb\sigma\rrb}\uh k}{\#T_\P\uh k}.
\end{equation}
By the VL-homogeneity of $\P$ and the fact that $\# T_{\P\cap\llb\sigma\rrb}\uh |\sigma|=1$, there is some constant $c$ such that
\[
\frac{1}{c}\cdot\frac{\# T_\P\uh k}{\# T_\P\uh |\sigma|} \leq\# T_{\P\cap\llb\sigma\rrb}\uh k\leq c\cdot\frac{\# T_\P\uh k}{\# T_\P\uh |\sigma|}.
\]
Dividing through by $\# T_\P\uh k$ yields
\begin{equation}\label{eq-yeah}
\frac{1}{c}\cdot\frac{1}{\# T_\P\uh |\sigma|} \leq\frac{\# T_{\P\cap\llb\sigma\rrb}\uh k}{\# T_\P\uh k}\leq c\cdot\frac{1}{\# T_\P\uh |\sigma|}.
\end{equation}
By applying (\ref{eq-another}) to (\ref{eq-yeah}), we get
\[
\frac{1}{c}\cdot\frac{1}{\# T_\P\uh |\sigma|} \leq\frac{\lambda(\llb\sigma\rrb\cap\llb T_\P\uh k\rrb)}{\lambda(T_\P\uh k)}\leq c\cdot\frac{1}{\# T_\P\uh |\sigma|}.
\]
In the case that $\lambda_\P(\sigma)=\lim_{k\rightarrow\infty}\frac{\lambda(\llb\sigma\rrb\cap\llb T_\P\uh k\rrb)}{\lambda(T_\P\uh k)}$ exists, we have
\[
\Bigl|-\log\lambda_\P(\sigma)-\log\#T_\P\uh|\sigma|\Bigr|\leq \log(c).
\]
In the case that $\lambda_\P(\sigma)$ is not defined, we have
\[
\frac{1}{c}\cdot\frac{1}{\# T_\P\uh |\sigma|} \leq\liminf_{k\rightarrow\infty}\frac{\lambda(\llb\sigma\rrb\cap\llb T_\P\uh k\rrb)}{\lambda(T_\P\uh k)}\leq\limsup_{k\rightarrow\infty}\frac{\lambda(\llb\sigma\rrb\cap\llb T_\P\uh k\rrb)}{\lambda(T_\P\uh k)}\leq c\cdot\frac{1}{\# T_\P\uh |\sigma|}
\]
from which we can derive
\[
\Bigl|-\log\lambda^+_\P(\sigma)-\log\#T_\P\uh|\sigma|\Bigr|\leq \log(c)\;\;\text{and}\;\;
\Bigl|-\log\lambda^-_\P(\sigma)-\log\#T_\P\uh|\sigma|\Bigr|\leq \log(c).
\]

\end{proof}

Using part (i) of Lemma \ref{lem-upperlower} we can now conclude the following:

\begin{theorem}\label{thm-vlhomequiv}
Let $\P$ be a VL-homogeneous $\Pi^0_1$ class and suppose that $\lambda_\P$ is defined.
 Then for any $X \in \P$, $X$ is globally random in $\P$ if and only if $X$ is $T_\P$-incompressible.

\end{theorem}

\subsection{Additive homogeneity} \label{sec6.4}

We lastly turn to an additive notion of homogeneity for $\Pi^0_1$ classes.

\begin{definition}  Let $\P$ be a $\Pi^0_1$ class. $\P$ is \emph{additively homogeneous} (\emph{a-homogeneous} for short) if there is some constant $c\in\omega$ such that for every $n$ and every $\sigma,\tau\in T_\P\uh n$,
\[
\bigl|\theta_\P(\sigma)-\theta_\P(\tau)\bigr|\leq c.
\]
\end{definition}



Recall Example \ref{ex7} above, where $\P = \{00,10,11\}^{\omega}$.  Here we have $\theta_\P(0^{2n}) = n$ but $\theta_\P(1^{2n}) = 2n$, so that $\P$ is not a-homogeneous.  Thus $n$-homogeneous $\Pi^0_1$ classes need not in general be a-homogeneous.  

We also have the following.

\begin{proposition}
There is an a-homogeneous $\Pi^0_1$ class $\P$ that is not weakly $n$-homogeneous for any $n\in\omega$.
\end{proposition}

\begin{proof}
We construct $\P$ level by level as follows.  For each $n\in\omega$, we define levels 
\[
T_\P\uh (n^2-n),T_\P\uh (n^2-n+1),\dotsc,T_\P\uh n^2
\]
to diagonalize against weak $n$-homogeneity.  We begin by including in $T_\P$ all strings of length 4 except for 1111 (and their initial segments).  This guarantees that $T_\P$ is not 2-homogeneous, as 00 has 4 extensions in $T_\P\uh 4$ but 11 only has 3.  For each string $\sigma\in T_\P\uh 4$ except for 1110, we add $\sigma0$ to $T_\P\uh 5$ as well as 11100 and 11101.  This has the effect that $\theta_\P(\sigma)=\theta_\P(\tau)$ for all $\sigma,\tau\in T_\P\uh 5$.  We continue a similar construction from levels 6 to 9, 12 to 16, 20 to 25, etc., while adding all possible extensions between these levels.  Clearly, the resulting $\Pi^0_1$ class is a-homogeneous (with constant $c=1$) and fails to be weakly $n$-homogeneous for all $n\in\omega$.
\end{proof}

This example also shows that so $\lambda_\P$ need not be defined for a-homogeneous classes. 

Despite the fact that $\lambda_\P$ is not defined for every a-homogeneous class $\P$, we still get a stronger result concerning the relationship between the various notions of randomness for $\Pi^0_1$ classes than we do for weakly $n$-homogeneous classes and VL-homogeneous classes, which we now show.  First we prove a useful combinatorial lemma.

\begin{lemma}\label{lem-branches}
Let $\P$ be a $\Pi^0_1$ class and let $n\in\omega$ and $\sigma\in\str$.
\begin{itemize}
\item[(i)] If every extension $\tau\succeq\sigma$ in $T_\P\uh n$ satisfies $\theta_\P(\tau)\geq\theta_\P(\sigma)+ \ell$, then $\sigma$ has at least $2^\ell$ extensions in $T_\P\uh n$.
\item[(ii)] If every extension $\tau\succeq\sigma$ in $T_\P\uh n$ satisfies $\theta_\P(\tau)=\theta_\P(\sigma)+ \ell$, then $\sigma$ has exactly $2^\ell$ extensions in $T_\P\uh n$.
\end{itemize}
\end{lemma}

\begin{proof}
(i) For $\ell=1$, given $\tau\succeq\sigma$ of length $n$ that satisfies $\theta_\P(\tau)\geq\theta_\P(\sigma)+ 1$, there must be some $j$ satisfying $|\sigma|<j<n$ such that $(\tau\uh j)^\curvearrowright\in T_\P$.  Hence there is some $\tau'\succeq(\tau\uh j)^\curvearrowright\succeq\sigma$ in $T_\P\uh n$, yielding two extensions of $\sigma$ of length $n$.  

Suppose the claim holds for all $n$ and $\sigma$ and some fixed $\ell$; we will show it similarly holds for $\ell+1$. Suppose  further that every extension $\tau\succeq\sigma$ of length $n$ satisfies $\theta_\P(\tau)\geq\theta_\P(\sigma)+ \ell+1$.  Let $\rho_0,\rho_1\succeq\sigma$ be the shortest incompatible extensions of $\sigma$.  Then for $i=0,1$, every extension of $\rho_i$ of length $n$ satisfies $\theta_\P(\tau)\geq\theta_\P(\sigma)+ \ell$.  Thus by the inductive hypothesis, $\rho_i$ has at least $2^\ell$ extensions of length $n$ for $i=0,1$.  Thus $\sigma$ has at least $2^{\ell+1}$ extensions of length $n$.

The proof of (ii) is nearly identical.
\end{proof}

\begin{lemma}\label{lem-ahom}
Let $\P$ be an a-homogeneous $\Pi^0_1$ class.  Then there is some $c\in\omega$ such that for every $n$ and every $\sigma\in T_\P\uh n$,
\[
\bigl|\theta_\P(\sigma)-\log\# T_\P\uh n\bigr|\leq c.
\]
\end{lemma}

\begin{proof}
Let $c$ satisfy $\bigl|\theta_\P(\sigma)-\theta_\P(\tau)\bigr|\leq c$ for every $n\in\omega$ and every $\sigma,\tau\in\# T_\P\uh n$.  For a fixed $n\in\omega$, let $m_n=\min\{\theta_\P(\sigma)\colon\sigma\in\# T_\P\uh n\}$.  Given $\sigma\in\# T_\P\uh n$ with $\theta_\P(\sigma)=m_n$, set $m=m_n$, and let $k\geq n$ be largest such that $\theta_\P(\sigma\uh k)=0$ (if $k=0$, $\sigma\uh k=\epsilon$). By the minimality of $m$, it follows that every extension $\tau\in T_\P\uh n$ of $\sigma\uh k$  satisfies $\theta_\P(\tau)\geq m=\theta_\P(\sigma)+m$.  By Lemma \ref{lem-branches}(i), $\sigma\uh k$, and hence $\sigma$, has $2^m$ extensions in $T_\P\uh n$.  Thus 
\begin{equation}\label{eq-branch1}
\# T_\P\uh n\geq 2^m=2^{m_n}.
\end{equation}
Next, let $M_n=\max\{\theta_\P(\sigma)\colon\sigma\in\ T_\P\uh n\}.$  In the most extreme case, for every $\tau\in T_\P\uh n$ we have $\theta_\P(\tau)=M_n$.  Applying the Lemma \ref{lem-branches}(ii) to $\epsilon$, this implies that $\epsilon$ has $2^{M_n}$ extensions in $T_\P\uh n$.  Thus,
\begin{equation}\label{eq-branch2}
\# T_\P\uh n\leq 2^{M_n}.
\end{equation}
Combining (\ref{eq-branch1}) and (\ref{eq-branch2}) and taking logarithms yields
\begin{equation}\label{eq-branch3}
m_n\leq \log\# T_\P\uh n\leq M_n.
\end{equation}
Given $\sigma\in T_\P\uh n$, it follows from our hypothesis that
\begin{equation}\label{eq-branch4}
M_n-c\leq \theta_\P(\sigma)\leq m_n+c.
\end{equation}
From (\ref{eq-branch3}) and (\ref{eq-branch4}) we can conclude
\[
\bigl|\theta_\P(\sigma)-\log\# T_\P\uh n\bigr|\leq c.
\]

\end{proof}

A careful reading of the proof of Lemma \ref{lem-ahom} shows that each step can be reversed, so that we have the following. 

\begin{corollary}
 A $\Pi^0_1$ class $\P$ is a-homogeneous if and only if there is some $c\in\omega$ such that for every $n$ and every $\sigma\in T_\P\uh n$,
\[
\bigl|\theta_\P(\sigma)-\log\# T_\P\uh n\bigr|\leq c.
\]
\end{corollary}

The following is an immediate consequence of Theorem \ref{thm-rpp-isc} and Lemma \ref{lem-ahom}.

\begin{theorem}\label{ahom-med}
Let $\P$ be an a-homogeneous $\Pi^0_1$ class.   Then $X\in \P$ is a randomly produced path through $T_\P$ if and only if $X$ is $T_\P$-incompressible.
\end{theorem}

Next we have:

\begin{theorem}\label{thm-ahom-svlhom}
Every a-homogeneous $\Pi^0_1$ class is VL-homogeneous.
\end{theorem}

\begin{proof}
Let $\P$ be an a-homogeneous $\Pi^0_1$ class.  By Lemma \ref{lem-vlhom2} it suffices to show that for all $n$, all $k\geq n$, and $\sigma\in T_\P\uh n$, there is some $d\in\omega$ such that
\[
\frac{1}{d}\cdot\frac{\#T_\P\uh k}{\#T_\P\uh n}\leq\#T_{\P\cap\llb\sigma\rrb}\uh k\leq d\cdot\frac{\#T_\P\uh k}{\#T_\P\uh n}.
\]
Fix such $n,k$ and $\sigma$.  Since $\P$ is a-homogeneous, by Lemma \ref{lem-ahom}, there is some $c$ such that
\[
\bigl|\theta_\P(\sigma)-\log\# T_\P\uh n\bigr|\leq c,
\]
from which it follows that
\begin{equation}\label{eq-ahom1}
2^{\theta_\P(\sigma)-c}\leq\# T_\P\uh n\leq2^{\theta_\P(\sigma)+c}.
\end{equation}
Now let 
\[
m_k=\min\{\theta_\P(\tau)-\theta_\P(\sigma)\colon \tau\succ\sigma \;\&\; \tau\in T_\P\uh k\}
\]
and
\[
M_k=\max\{\theta_\P(\tau)-\theta_\P(\sigma)\colon \tau\succ\sigma \;\&\; \tau\in T_\P\uh k\}.
\]
Then since every extension $\tau$ of $\sigma$ in $T_\P\uh k$ satisfies $\theta_\P(\tau)\geq\theta_\P(\sigma)+m_k$, by Lemma \ref{lem-branches}(i), we have $\#T_{\P\cap\llb\sigma\rrb}\uh k\geq 2^{m_k}$.  Moreover, since every extension $\tau$ of $\sigma$ in $T_\P\uh k$ satisfies $\theta_\P(\tau)\leq \theta_\P(\sigma)+M_k$, in the most extreme case, we have $\theta_\P(\tau)= \theta_\P(\sigma)+M_k$ for all such $\tau$.  It then follows from Lemma \ref{lem-branches}(ii) that $\#T_{\P\cap\llb\sigma\rrb}\uh k= 2^{M_k}$.  Given that this is the most extreme case, it follows that $\#T_{\P\cap\llb\sigma\rrb}\uh k\leq 2^{M_k}$ in general.   Summing up, we have
\begin{equation}\label{eq-ahom2}
2^{m_k}\leq \#T_{\P\cap\llb\sigma\rrb}\uh k\leq 2^{M_k}.
\end{equation}
Note by a-homogeneity, it follows that $M_k\leq m_k+c$, which combined with (\ref{eq-ahom2}) yields
\begin{equation}\label{eq-ahom2.5}
2^{m_k}\leq \#T_{\P\cap\llb\sigma\rrb}\uh k\leq 2^{m_k+c}.
\end{equation}
Next, if $\tau\in T_{\P\cap\llb\sigma\rrb}\uh k$, then by a-homogeneity and the definition of $m_k$ and $M_k$, we have
\begin{equation}\label{eq-ahom3}
\theta_\P(\sigma)+m_k-c\leq\theta_\P(\tau)\leq\theta_\P(\sigma)+M_k+c.
\end{equation}
Again, by a-homogeneity, we have
\[
\bigl|\theta_\P(\tau)-\log\# T_\P\uh k\bigr|\leq c.
\]
from which it follows that
\[
2^{\theta_\P(\tau)-c}\leq\# T_\P\uh k\leq2^{\theta_\P(\tau)+c}.
\]
Combined with (\ref{eq-ahom3}), this yields
\[
2^{\theta_\P(\sigma)+m_k-2c}\leq\# T_\P\uh k\leq2^{\theta_\P(\sigma)+M_k+2c}.
\]
and hence
\begin{equation}\label{eq-ahom4}
2^{\theta_\P(\sigma)+m_k-2c}\leq\# T_\P\uh k\leq2^{\theta_\P(\sigma)+m_k+3c},
\end{equation}
since $M_k\leq m_k+c$. Now we have by (\ref{eq-ahom1}) and (\ref{eq-ahom4})
\[
\frac{2^{\theta_\P(\sigma)+m_k-2c}}{2^{\theta_\P(\sigma)+c}}\leq\frac{\# T_\P\uh k}{\# T_\P\uh n}\leq\frac{2^{\theta_\P(\sigma)+m_k+3c}}{2^{\theta_\P(\sigma)-c}},
\]
which simplifies to
\begin{equation}\label{eq-ahom5}
2^{m_k-3c}\leq\frac{\# T_\P\uh k}{\# T_\P\uh n}\leq 2^{m_k+4c}.
\end{equation}
Then by the second inequality in (\ref{eq-ahom2.5}) and the first inequality in (\ref{eq-ahom5})
\[
\#T_{\P\cap\llb\sigma\rrb}\uh k\leq 2^{m_k+c}\leq 2^{4c}\frac{\# T_\P\uh k}{\# T_\P\uh n}.
\]
Similarly, by the first inequality in (\ref{eq-ahom2.5}) and the last inequality in (\ref{eq-ahom5})
\[
\frac{\# T_\P\uh k}{\# T_\P\uh n}\leq 2^{m_k+4c}\leq 2^{4c}\#T_{\P\cap\llb\sigma\rrb}\uh k,
\]
which, when setting $d=2^{4c}$, establishes the conclusion.
\end{proof}

\begin{theorem}\label{thm-big}
Let $\P$ be an a-homogeneous $\Pi^0_ 1$ class and suppose that $\lambda_\P$ is defined for all $\sigma\in T_\P$.  Then for $X\in\P$, the following are equivalent:
\begin{itemize}
\item[(a)] $X$ is a randomly produced path through $T_\P$.
\item[(b)] $X$ is $T_\P$-incompressible.
\item[(c)] $X$ is globally random in $\P$.
\end{itemize}
\end{theorem}

\begin{proof}
This follows immediately from Theorems \ref{thm-vlhomequiv} and \ref{thm-ahom-svlhom}.
\end{proof}

In the case that $\lambda_\P$ is not defined for an a-homogeneous class $\P$, we still can prove the following:

\begin{theorem}
Let $\P$ be an a-homogeneous $\Pi^0_ 1$ class.  Then for $X\in\P$, the following are equivalent:
\begin{itemize}
\item[(a)] $X$ is a randomly produced path through $T_\P$.
\item[(b)] $X$ is $T_\P$-incompressible.
\item[(c)] $K^{T_\P}(X\uh n)\geq -\log\lambda^+_\P(X\uh n)-O(1)$ for all $n$.
\item[(d)] $K^{T_\P}(X\uh n)\geq -\log\lambda^-_\P(X\uh n)-O(1)$ for all $n$.
\end{itemize}
\end{theorem}

\begin{proof}
If $\lambda_\P$ is defined, then $\lambda_\P^+=\lambda_\P^-$ and the result immediately follows from Theorem \ref{thm-big}.  If $\lambda_\P$ is not defined, then by Lemma \ref{lem-upperlower}(ii), we have
\[
\Bigl|-\log\lambda_\P^+(\sigma)-\log\#T_\P\uh|\sigma|\Bigr|\leq O(1)
\]
and
\[
\Bigl|-\log\lambda_\P^-(\sigma)-\log\#T_\P\uh|\sigma|\Bigr|\leq O(1).
\]
In particular, we also have
\[
\Bigl|-\log\lambda_\P^+(\sigma)-(-\log\lambda_\P^-(\sigma))\Bigr|\leq O(1).
\]
The equivalence then follows from the above inequalities and Theorem \ref{ahom-med}.
\end{proof}

The relationship between the various notions of homogeneity can be summed up in Figure \ref{fig1} (where the absence of an implication arrow means that implication in question does not hold). In addition, the main properties of the various notions of homogeneity, whether $\lambda_\P$ is always defined and whether all randomness notions coincide in the case that $\lambda_\P$ is defined, are summed up in Table \ref{table1}

\begin{figure}[h]
\captionsetup{justification=centering}
\begin{center}
\includegraphics[scale=1.0]{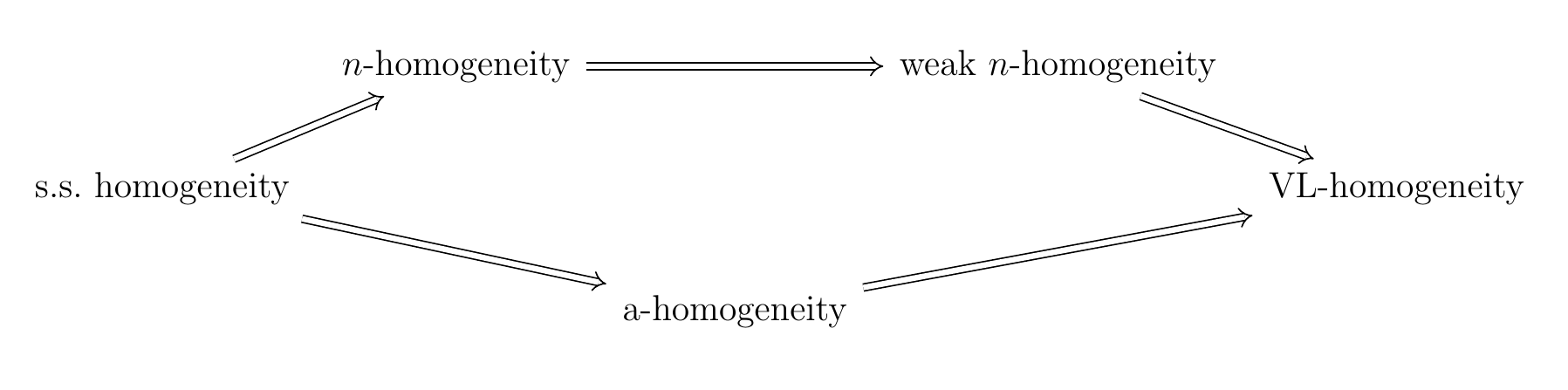}
\end{center}
\label{fig1}
\caption{The relationship between notions of homogeneity for $\Pi^0_1$ classes}
\end{figure}


\begin{table}
\centering
\begin{tabular}{|c|c|c|}
\hline
& $\lambda_\P$ always defined & all randomness notions coincide\\
\hline
s.s.\ homogeneity & Y & Y \\
\hline
$n$-homogeneity& N & N\\
\hline
weak $n$-homogeneity& N & N\\
\hline
a-homogeneity & N & Y\\
\hline
VL-homogeneity & N & N\\
\hline
\end{tabular}
\caption{Main properties of the various notions of homogeneity}
\label{table1}
\end{table}

\section{$\Pi^0_1$ classes of positive measure} \label{sec-Positive}

We conclude with a discussion of notions of random members of certain $\Pi^0_1$ classes of positive measure.
We begin with a basic fact.

\begin{lemma} \label{lem9}
If $\P$ is a $\Pi^0_1$ class of positive measure, then $\P$ satisfies the condition ($\dagger$) in the definition of VL-homogeneity.
\end{lemma}

\begin{proof}
Let $\lambda(\P)>2^{-c}$.
We claim that for every $n,m$,
\[
\dfrac{\# T_\Q\uh (n+m)}{\# T_\Q\uh n}\leq 2^c\cdot\dfrac{\# T_\P\uh (n+m)}{\# T_\P\uh n}.
\]
To show this, we will prove that for every $n,m$,
\begin{equation}\label{eq-posmeas}
\#T_{\P}\uh (n+m)\geq 2^{m-c}\#T_{\P}\uh n.
\end{equation}
Suppose not.  Then there are $n,m$ such that 
\[
\#T_{\P}\uh (n+m)<2^{m-c}\#T_{\P}\uh n.
\]
Then, noting that for any $\Pi^0_1$ class $\Q$, we have $\lambda(\Q)\leq 2^{-k}\cdot\#T_{\Q}\uh k$ for every $k\in\omega$, it follows that
\[
\lambda(\P)\leq 2^{-(n+m)}\cdot\#T_{\P}\uh (n+m)<2^{-(n+c)}\cdot\#T_{\P}\uh n\leq 2^{-(n+c)}\cdot 2^n=2^{-c},
\]
which contradicts our original assumption.  The conclusion now follows from (\ref{eq-posmeas}) and the fact that
\[
\dfrac{\# T_\Q\uh (n+m)}{\# T_\Q\uh n}\leq 2^m.
\]
\end{proof}

 For an example of a $\Pi^0_1$ class of positive measure that fails to satisfy condition ($\dagger\dagger$), consider $\P\oplus \Q$, where $\P=\cs$ and $\Q=\{0^\omega\}$.
 
We now can consider the random members of various $\Pi^0_1$ classes of positive measure.  Our discussion makes use of the notion of \emph{lower dyadic density}. We define the lower dyadic density of $\P$ at $X$ to be
\[
\rho(\P\mid X)=\liminf_{n\rightarrow\infty}\lambda(\P\mid X\uh n)=\liminf_{n\rightarrow\infty}\dfrac{\lambda(\P\cap\llb X\uh n\rrb)}{\lambda(X\uh n)}.
\]
We note that the term ``lower density" is used in the context of effectively closed subsets of $[0,1]$, while ``lower dyadic density" is used in the contexts of effective closed subsets of $\cs$. The lower density of a point in a given effectively closed class has been of considerable interest recently.  See, for instance, \cite{BHMN14}, \cite{BGKNT16}, and \cite{K16}.


Recall that $\MLR^{T_\P}$ is simply the collection of sequences that are $\lambda$-Martin-L\"of random relative to $T_\P$.

\begin{theorem}\label{thm-pos}
Let $\P$ be a $\Pi^0_1$ class of positive measure.  For $X$ of positive lower dyadic density in $\P$, the following are equivalent:
\begin{itemize}
\item[(i)]  $X\in\MLR^{T_\P}$;
\item[(ii)] $K^{T_\P}(X\uh n)\geq -\log\lambda_\P(X\uh n)-0(1)$ for all $n\in\omega$;
\item[(iii)] $K^{T_\P}(X\uh n)\geq \log \#T_\P\uh n-O(1)$ for all $n\in\omega$.
\end{itemize}
\end{theorem}

\begin{proof}
(i) $\Rightarrow$ (ii) Since $X$ has positive lower dyadic density in $\P$, there is some $c$ and some $N$ such that for all $n\geq N$, 
\[
\frac{\lambda(\llb X\uh n\rrb\cap\P)}{\lambda(X\uh n)}>2^{-c}
\]
which implies that $\lambda(\llb X\uh n\rrb\cap\P)\geq 2^{-(n+c)}$ for almost every $n$.  From this we can conclude that
\[
\lambda_\P(X\uh n)=\frac{\lambda(\llb X\uh n\rrb\cap\P)}{\lambda(\P)}\geq \frac{2^{-(n+c)}}{\lambda(\P)}
\]
and thus
\[
n\geq -\log\lambda_\P(X\uh n)-O(1)
\]
for every $n$.  Then by the relativized Levin-Schnorr theorem,
\[
K^{T_\P}(X\uh n)\geq n-O(1)\geq -\log\lambda_\P(X\uh n)-O(1).
\]
(ii)$\Rightarrow$(iii)  Since $\lambda(\P)>0$, it follows from Lemma \ref{lem9}  that $\P$ satisfies condition ($\dagger$) in the definition of VL-homogeneity.  By the proof of Lemma \ref{lem-upperlower} , the satisfaction of condition ($\dagger$) implies that
\[
\lambda( X\uh n\mid\P)\leq c\cdot\dfrac{1}{\# T_\P\uh n}
\]
for some $c\in\omega$.  This implies that 
\[
-\log\lambda_\P(X\uh n)\geq \log\#T_\P\uh n-O(1)
\]
for every $n$, from which the desired implication immediately follows.

(iii)$\Rightarrow$(i) Let $c\in\omega$ satisfy $\lambda(\P)>2^{-c}$.  For $n\in\omega$, since $\lambda(\P)\leq 2^{-n}\cdot\#T_{\P}\uh n$, it follows that $\#T_{\P}\uh n\geq 2^{n-c}$.  Thus $\log\#T_\P\uh n\geq n-c$ and hence
\[
K^{T_\P}(X\uh n)\geq \log\#T_\P\uh n-O(1)\geq n-O(1),
\]
from which it follows that $X$ is $T_\P$-random.

\end{proof}

By Propositions \ref{prop-pos1} and \ref{prop-pos2} below, we need an additional hypothesis on $X$ to include the property of being a randomly produced path through $T_\P$ with the conditions (i)-(iii) in Theorem \ref{thm-pos}.

A number of corollaries follow immediately from Theorem \ref{thm-pos}.

\begin{corollary}\label{thm-2ran}
Let $\P$ be a decidable $\Pi^0_1$ class. For $X$ of positive lower dyadic density in $\P$, the following are equivalent:
\begin{itemize}
\item[(i)] $X$ is Martin-L\"of random;
\item[(ii)] $K(X\uh n)\geq -\log\lambda_\P(X\uh n)-0(1)$ for all $n\in\omega$;
\item[(iii)] $K(X\uh n)\geq \log \#T_\P\uh n-O(1)$ for all $n\in\omega$.
\end{itemize}
\end{corollary}

\begin{corollary}\label{cor-2ran}
Let $\P$ be a $\Pi^0_1$ class of positive measure such that $T_\P\equiv_T\emptyset'$.   For $X\in\P$, the following are equivalent:
\begin{itemize}
\item[(i)] $X$ is 2-random, i.e., $X\in\MLR^{\emptyset'}$;
\item[(ii)] $K^{T_\P}(X\uh n)\geq -\log\lambda_\P(X\uh n)-0(1)$ for all $n\in\omega$;
\item[(iii)] $K^{T_\P}(X\uh n)\geq \log \#T_\P\uh n-O(1)$ for all $n\in\omega$.
\end{itemize}
\end{corollary}

\begin{proof}
It follows from work of Bienvenu et al. in \cite{BHMN14} that if $X$ is 2-random and $X\in\P$, $X$ has lower density 1 in $\P$; moreover, Khan and Miller \cite{K16} showed that for Martin-L\"of random sequences, having lower density 1 and lower dyadic density one are equivalent.  Thus $X$ satisfies the hypothesis of Theorem \ref{thm-pos} and the conclusion follows.
\end{proof}

\begin{corollary}
Let $\P=\cs\setminus \hat{\U_i}$, where $(\hat{\U_i})_{i\in\omega}$ is a universal Martin-L\"of test.  Then for $X\in\P$, the conditions (i)-(iii) in Corollary \ref{cor-2ran} are equivalent.
\end{corollary}

\begin{proof}
Since $T_\P$ is co-c.e.\ and $T_\P\geq_T X$ for some $X\in\MLR$ (such as the leftmost path of $\P$), $T_\P$ has diagonally non-computable (DNC) degree.  That is, $T_\P$ computes a total function $f$ such that for all $i$, if $\phi_i(i)\halts$ then $f(i)\neq \phi_i(i)$.  By Arslanov's Completeness Criterion, the only c.e.\ DNC degree is the complete one, hence $T_\P\equiv_T\emptyset'$.   
The result thus follows by applying Corollary \ref{cor-2ran}. 
\end{proof}

The branching number function $\theta_\P$  played an important role in Section \ref{sec-Homogeneous} in characterizing the randomness of members of classes which satisfied various forms of homogeneity. 
We now investigate the connection between density in a $\Pi^0_1$ class $\P$ of positive measure and the function $\theta_\P$.  First we prove the following general lemma about lower density and $\theta_\P$ in an arbitrary $\Pi^0_1$ class $\P$.

\begin{lemma}\label{lem-density}
Let $\P$ be a $\Pi^0_1$ class of positive measure and let $X\in\P$.  Suppose that
$\rho(\P\mid X)>2^{-k}$ for some $k\in\omega$. Then $\theta_\P(X\uh n)\geq \dfrac{n}{k}-O(1)$ for all $n\in\omega$.
\end{lemma}

\begin{proof}
Since
\[
\liminf_{n\rightarrow\infty}\lambda(\P\mid X\uh n)>2^{-k},
\]
there is some $N_k$ such that for all $n\geq N_k$
\[
\lambda(\P\mid X\uh n)>2^{-k}.
\]
Suppose now that there is some $\ell\geq N_k$ such that for every $i\in\{1,\dotsc k\}$ we have $X\uh(\ell+i)\in T_\P$ but $(X\uh (\ell+i))^\curvearrowright\notin T_\P$.  For $j\in\omega$, observe that $(X\uh j)^\curvearrowright\notin T_\P$ implies that $\P\cap\llb (X\uh j)^\curvearrowright\rrb=\emptyset$ and thus for every $i\in\{1,\dotsc,k\}$,
\begin{equation}\label{eq7}
\begin{split}
\lambda\bigl(\P\cap\llb X\uh (\ell+i-1)\rrb\bigr)
&=\lambda\bigl(\P\cap\llb X\uh (\ell+i)\rrb\bigr)+\lambda\bigl(\P\cap\llb (X\uh (\ell+i))^\curvearrowright\rrb\bigr)\\
&=\lambda\bigl(\P\cap\llb X\uh (\ell+i)\rrb\bigr).
\end{split}
\end{equation}
It follows from (\ref{eq7}) that for every $i\in\{1,\dotsc,k\}$,
\begin{equation}\label{eq8}
\dfrac{\lambda(\P\cap\llb X\uh (\ell+i)\rrb)}{\lambda(X\uh (\ell+i))}=
\dfrac{\lambda(\P\cap\llb X\uh (\ell+i-1)\rrb)}{\lambda(X\uh (\ell+i))}=
2\cdot\dfrac{\lambda(\P\cap\llb X\uh (\ell+i-1)\rrb)}{\lambda(X\uh (\ell+i-1))}.
\end{equation}
Applying (\ref{eq8}) $k$ times to $\lambda(\P\cap\llb X\uh (\ell+k)\rrb)/\lambda(X\uh (\ell+k))$ yields
\begin{equation}\label{eq9}
\dfrac{\lambda(\P\cap\llb X\uh (\ell+k)\rrb)}{\lambda(X\uh (\ell+k))}=
2^k\cdot\dfrac{\lambda(\P\cap\llb X\uh \ell\rrb)}{\lambda(X\uh \ell)}.
\end{equation}
Next, by choice of $\ell$ we have
\begin{equation}\label{eq10}
2^{-k}<\dfrac{\lambda(\P\cap\llb X\uh \ell\rrb)}{\lambda(X\uh \ell)}\leq 1.
\end{equation}
Combining (\ref{eq9}) and (\ref{eq10}) yields
\[
1<\dfrac{\lambda(\P\cap\llb X\uh (\ell+k)\rrb)}{\lambda(X\uh (\ell+k))}\leq 2^k,
\]
which is clearly impossible.  Thus, for each $\ell\geq N_k$ and each block of $k$ values $\ell+1,\dotsc,\ell+k$, there must be some $i\in\{1,\dotsc,k\}$ such that both $X\uh(\ell+i)\in T_\P$ and $(X\uh (\ell+i))^\curvearrowright\in T_\P$.  Thus there is some sufficiently large $N$ such that $\theta_\P(X\uh n)\geq \dfrac{n}{k}-N$ for every $n\geq N$.
\end{proof}

\begin{proposition}\label{prop-pos1}
Let $\P$ be a $\Pi^0_1$ class of positive measure.  If $X\in\P\cap\MLR^{T_\P}$, then $X$ is a randomly produced path through $T_\P$.  In particular, if $T_\P\equiv_T\emptyset'$, then every 2-random $X\in\P$ is a randomly produced path through $T_\P$.
\end{proposition}

\begin{proof}
Since $\theta_\P(\sigma)\leq |\sigma|$ for every $\sigma\in\str$ and $X\in\MLR^{T_\P}$, we have
\[
K^{T_\P}(X\uh n)\geq n-O(1)\geq \theta_\P(X\uh n)-O(1).
\]
Thus $X$ is a randomly produced path through $T_\P$.
\end{proof}

\begin{proposition}\label{prop-pos2}
Let $\P$ be a $\Pi^0_1$ class of positive measure. If $X\in\P$ is a randomly produced path through $T_\P$ of lower dyadic density in $\P$ strictly greater than 1/2, then $X\in\MLR^{T_\P}$. In particular, if $T_\P\equiv_T\emptyset'$, then every randomly produced path $X$ through $\P$ of lower dyadic density in $\P$ strictly greater than 1/2 is 2-random.
\end{proposition}

\begin{proof}
Since $X$ has has lower dyadic density strictly greater than 1/2 in $\P$, it follows from Lemma \ref{lem-density} that $\theta_\P(X\uh n)\geq n-O(1)$ for every $n\in\omega$.  Then
\[
K^{T_\P}(X\uh n)\geq \theta_\P(X\uh n)\geq n-O(1).
\]
Thus $X\in\MLR^{T_\P}$.
\end{proof}

Note that the density condition is necessary:  Let $\P=\cs$ and $\Q$ be a $\Pi^0_1$ class containing no 1-randoms (such as $\{0^\omega\}$ or $\{00,11\}^\omega$).  Then not every randomly produced path through $T_{\P\oplus\Q}$ is in $\MLR^{T_\P}$ (i.e., those randomly produced paths extending the string 1).

By Theorem \ref{thm-pos} and Propositions \ref{prop-pos1} and \ref{prop-pos2} we have the following.

\begin{corollary}
Let $\P$ be a $\Pi^0_1$ class of positive measure.  For any $X$ with lower dyadic density in $\P$ strictly greater than 1/2, the following are equivalent:
\begin{itemize}
\item[(i)]  $X\in\MLR^{T_\P}$;
\item[(ii)] $K^{T_\P}(X\uh n)\geq -\log\lambda_\P(X\uh n)-0(1)$ for all $n\in\omega$;
\item[(iii)] $K^{T_\P}(X\uh n)\geq \log \#T_\P\uh n-O(1)$ for all $n\in\omega$;
\item[(iv)] $X$ is a randomly produced path through $T_\P$.
\end{itemize}
\end{corollary}

One interesting consequence of this result is that if $\P=\cs\setminus\hat{\U_i}$, where $(\hat{\U_i})_{i\in\omega}$ is a universal Martin-L\"of test, then if $X$ has lower dyadic density in $\P$ strictly greater than 1/2 and is not 2-random, then $X$ is not a random member of $\P$ (according to any of our definitions).  As shown by Bienvenu et al. \cite{BGKNT16}, there are sequences $X$ that have lower density 1 in every effectively closed class containing $X$ (which, by the result of Khan and Miller \cite{K16} referenced above, also holds of lower dyadic density) but which are not 2-random (they exhibit a low Oberwolfach random sequence, where Oberwolfach randomness is a notion of randomness that guarantees lower density 1).  However, the following question remains open.

\begin{question}
If $\P$ is a $\Pi^0_1$ class of positive measure, are there randomly produced paths through $\P$ with lower dyadic density at most 1/2?  In particular, if $T_\P\equiv\emptyset'$, is $X\in\P$ a randomly produced path through $T_\P$ if and only if $X$ is 2-random? 
\end{question}

\section{Conclusion and future work}\label{sec-future}

In this paper, we define some notions of randomness for members of a given closed set $\P$. $\P$ may be defined as the set of infinite paths through a tree $T_{\P}$ with no dead ends.  We define a natural mapping $\Psi_{\P}$ from $\cs$ into $\P$ which uses an input $Y$ to determine which branch to follow in $T_{\P}$, and say that $X \in \P$ is \emph{randomly produced path through $T_\P$} if $X = \Phi_{\P}(R)$ for some Martin-L\"of random sequence $R$.  The map $\Psi$ induces a measure on $\P$ and we show that $X$ is a randomly produced member of $\P$ if and only if it is $T_{\P}$ random with respect to this measure.  We give an alternative mapping $\Phi$ and show that randomness with respect to the measure induced by $\Phi$  is equivalent to being randomly produced.  The branching number $\theta_{\P}(\sigma)$ is defined to be the number of initial segments $\tau$ of $\sigma$ such that both $\sigma \fr 0$ and $\sigma \fr 1$ are in $T_{\P}$. 
It is shown that $X \in \P$ is a randomly produced element if and only if  $K^{T_\P}(X\uh n)\geq \theta_\P(X\uh n)-O(1)$, which is a notion of incompressibility. 

We define a  \emph{limiting relative measure}  $\lambda_\P(\sigma)=\lim_{n\rightarrow\infty}\frac{\lambda(\llb\sigma\rrb\cap \llb T_\P\uh n\rrb)}{\lambda(T_\P\uh n)}$.  This will be the usual relative Lebesgue measure if $\lambda(\P) > 0$ but otherwise may not exist. We say that $X \in \P$ is \emph{globally random} if it is $T_{\P}$-Martin-L\"of random with respect to the measure $\lambda_P$.  We say that $X \in \P$ is \emph{$T_{\P}$-incompressible} if  $K^{T_\P}(X\uh n)\geq log \# T_{\P} \uh n)-O(1)$.

Several notions of homogeneity for $\Pi^0_1$ classes are developed in part to find conditions under which $\lambda_P$ is defined and to characterize randomness with respect to $\lambda_P$. 
The closed set $\P$ is said to be \emph{s.s.-homogeneous} if for every $\sigma, \tau \in T_{\P}$ of the same length, $\sigma \fr i \in T_{\P} \iff \tau \fr \i \in T_{\P}$ for $i \in \{0,1\}$.  We show that if $\P$ is s.s.-homogeneous, then 
$\lambda_P$ is defined and for all $X \in \P$, $X$ is a randomly produced path if and only if $X$ is globally random in $\P$ if and only if $X$ is $T_{\P}$ incompressible. This result is extended to \emph{$n$-homogeneous} classes $\P$, where all nodes of length $nk$ have the same extensions of length $n(k+1)$ in $T_{\P}$. We also say that $\P$ is \emph{weakly $n$-homogeneous} if each node in $T_\P$ of length $nk$ has the same number of extensions of length $n(k+1)$.

Van Lambalgen develop a notion of multiplicative homogeneity for a $\Pi^0_1$ class $\P$, which concerns the relative number of branches in any $\Pi^0_1$ subclass of $\P$. We show that any weakly $n$-homogeneous class is VL-homogeneous; we call this notion \emph{VL-homogeneity}.  We show that if $\lambda_P$ is defined for a VL-homogeneous $\Pi^0_1$ class, then 
$\Bigl|-\log\lambda_\P(\sigma)-\log\#T_\P\uh|\sigma|\Bigr|\leq O(1),$ from which it follows that $X \in \P$ is globally random in $\P$ if and only if $X$ is $T_\P$-incompressible. 

For a final notion of homegeneity, we say that $\P$ is \emph{additively homogeneous} if there is a constant $c$ such that for all $n$ and all $\sigma,\tau \in T_{\P}$ of length $n$, $|\theta_P(\sigma) - \theta_P(\tau)| \leq c$.  We show that a member $X$ of an additively homogeneous $\Pi^0_1$ class $\P$ is a randomly produced path through $T_\P$ if and only if it is $T_{\P}$-incompressible. We further show that every additively homogeneous $\Pi^0_1$ class is VL-homogeneous.  From this it follows that if $\lambda_\P$, then each of the main notions of randomness coincide in $\P$.

Lastly, we examine the random members of $\Pi^0_1$ classes of positive measure, isolating conditions under which all of the main notions of randomness coincide for such classes.

In future work we plan to investigate the random members of deep $\Pi^0_1$ classes \cite{BP16} (such as the collection of completions of Peano arithmetic and the collection of $\{0,1\}$-valued diagonally non-computable functions, which is an s.s.-homogeneous class), thin $\Pi^0_1$ classes, computably perfect $\Pi^0_1$ classes, and ranked classes.

\bibliographystyle{plain}
\bibliography{random}

\end{document}